\documentclass[12pt]{article}
\usepackage[margin=1.0in]{geometry}
\usepackage{amsthm,amsmath,amssymb, amsfonts}
\usepackage{color}

\usepackage[T1]{fontenc}
\usepackage{charter}

\usepackage{setspace}
\def\bP{\mathbb P}
 
\def\R{\mathbb{R}}

\def\sF{{\mathcal F}}

\def\sA{{\mathcal A}}
\def\sL{{\mathcal L}}

\def\sS{{\mathcal S}}

\def\bL{\mathbb{L}}

\def\E{\mathbb{E}}

\def\sF{\mathcal{F}}

\def\bQ{\mathbb{Q}}

\def\eps{\varepsilon}
\newcommand{\bF}{\mathbb{F}}

\newcommand{\bmo}{\text{bmo}}

\newcommand{\ybar}{\overline{y}}
\newcommand{\zbar}{\overline{z}}

\numberwithin{equation}{section}
\theoremstyle{plain}                
\newtheorem{theorem}{Theorem}[section]
\newtheorem{lemma}[theorem]{Lemma}

\theoremstyle{definition}           
\newtheorem{definition}[theorem]{Definition}

\newtheorem{assumption}[theorem]{Assumption}
\theoremstyle{remark}               
\newtheorem{remark}{Remark}[section]

\begin{document}

\pagenumbering{arabic} \pagestyle{plain}
\begin{center}
  {\large \bf Existence of an equilibrium with limited participation}
  \ \\ \ \\
Kim Weston\footnote{The author acknowledges support by the National Science Foundation under Grant No. DMS\#1908255 (2019-2022). Any opinions, findings and conclusions or recommendations expressed in this material are those of the author and do not necessarily reflect the views of the National Science Foundation (NSF).}\\

Rutgers University
\ \\


\today
\end{center}
\vskip .5in


A limited participation economy models the real-world phenomenon that some economic agents have access to more of the financial market than others.  We prove the global existence of a Radner equilibrium with limited participation, where the agents have exponential preferences and derive utility from both running consumption and terminal wealth.  Our analysis centers around the existence and uniqueness of a solution to a coupled system of quadratic backward stochastic differential equations (BSDEs).  We prove that the BSDE system has a unique $\sS^\infty\times\bmo$ solution.  We define a candidate equilibrium in terms of the BSDE solution and prove through a verification argument that the candidate is a Radner equilibrium with limited participation.  This work generalizes the model of Basak and Cuoco~\cite{BC98RFS} to allow for a stock with a general dividend stream and agents with exponential preferences.  We also provide an explicit example.


\section{Introduction}\label{section:intro}
We study a financial equilibrium with limited participation, in which the economic agents do not have equal access to trade in the financial market.  We prove the existence of a Radner equilibrium in a model with two exponential agents, where one agent is excluded from the stock market.  The agents derive utility from running and terminal consumption with a finite time horizon in a pure-exchange economy.

The problem of a limited participation economy dates back to Basak and Cuoco~\cite{BC98RFS}, who introduced a continuous-time, running consumption model of limited participation with two (classes of) economic agents:  an unconstrained agent with access to a complete market, and a constrained agent who cannot trade in the stock market and faces incompleteness.  The equilibrium existence result in Basak and Cuoco~\cite{BC98RFS} assumes that both agents have logarithmic preferences.  Although logarithmic preferences are mathematically convenient, they are practically limiting and undesirable.  In an extension to Basak and Cuoco~\cite{BC98RFS} on an infinite time horizon, Prieto~\cite{P13wp} proves equilibrium existence with limited participation when the constrained agent is logarithmic and the unconstrained agent has power preferences.  Chabakauri~\cite{C15JME} provides heuristics and numerical results for the case of two agents with power preferences but does not prove the existence of an equilibrium with limited participation in his setting.  We seek to prove the existence of an equilibrium with limited participation in a setting similar to Basak and Cuoco~\cite{BC98RFS} that moves beyond the assumption of logarithmic preferences for the constrained agent.

Our extension to the model of Basak and Cuoco~\cite{BC98RFS} allows for a general dividend stream for the stock and investors with preferences described by exponential utility functions.  We focus our analysis on a three-dimensional system of quadratic, coupled BSDEs.  
Our contribution is three-fold.  First, we derive a quadratic, coupled system of BSDEs as a candidate to describe an equilibrium with limited participation.  Next, given an $\sS^\infty\times\bmo$ solution to the BSDE system, we construct an equilibrium and prove that our construction satisfies the definition of equilibrium in a verification argument.  Finally, we prove the existence and uniqueness of an $\sS^\infty\times\bmo$ solution to the BSDE system, which shows that a Radner equilibrium exists.

BSDEs, and in particular quadratic BSDEs, have a rich history in applications to mathematical finance.  For example, El Karoui et.\,al.~\cite{EKPQ97MF} price contingent claims via a BSDE, Hu et.\,al.~\cite{HIM05AAP} study utility maximization using BSDEs, Frei and dos Reis~\cite{FDR11MFE} and Espinosa and Touzi~\cite{ET15MF} consider utility maximization with relative performance features, and Kramkov and Pulido~\cite{KP16AAP} study a quadratic BSDE system derived from a price impact model.  Coupled systems of quadratic BSDEs arise when multiple agents interact through a financial market.  Kardaras et.\,al.~\cite{KXZ21wp}, Xing and \v{Z}itkovi\'c~\cite{XZ18AP}, Weston and \v{Z}itkovi\'c~\cite{WZ20FS}, and Escauriaza et.\,al.~\cite{ESX21AAP} study Radner equilibrium problems where the agents have exponential preferences and the economy is supported by an {\it incomplete} financial market.  Such incomplete financial equilibria are notoriously difficult to study.  Our model contributes to this literature, since our constrained agent faces incompleteness.

Existence of a solution to a system of quadratic BSDEs is delicate.  Results such as Tevzadze~\cite{T08SPA}, Cheridito and Nam~\cite{CN15S}, and Hu and Tang~\cite{HT16SPA} require a special structure to apply, while the two-dimensional counterexample in Frei and dos Reis~\cite{FDR11MFE} shows the need for at least some additional structure on quadratic BSDE systems to guarantee existence of a solution.  Xing and \v{Z}itkovi\'c~\cite{XZ18AP} study BSDE systems with triangularly quadratic drivers and an underlying Markov structure.  Our BSDE system, defined below in \eqref{def:bsde}, has the triangularly quadratic form as required by Xing and \v{Z}itkovi\'c~\cite{XZ18AP}.  However, a careful look at our system reveals that a linear transformation of an autonomous subsystem is diagonally quadratic, allowing the results of Fan et.\,al.~\cite{FHT20wp} to apply after suitable truncation arguments and estimates.  The diagonally quadratic driver structure and the results of Fan et.\,al.~\cite{FHT20wp} further allow us to achieve uniqueness within the class of $\sS^\infty\times\bmo$ solutions.

Two important model features are key to deriving our BSDE system: exponential preferences and a traded annuity.  An annuity is an asset that pays a constant dividend stream of one per unit time.  Annuities are commonplace in financial economics and actuarial sciences.  In mathematical finance, Christensen and Larsen~\cite{CL14RAPS} introduced a traded annuity in an incomplete Radner equilibrium with running consumption and stochastic interest rates.  With stochastic interest rates (or even stochastic shadow interest rates), the traded annuity allows for a variable reduction in an individual agent's value function.  In another incomplete equilibrium model with running consumption, Weston and \v{Z}itkovi\'c~\cite{WZ20FS} identified the annuity variable reduction in the individual agent problems that we will exploit in the present work.

The paper is organized as follows.  Section~\ref{section:set-up} sets up the model and defines a Radner equilibrium with limited participation.  Section~\ref{section:main-results} describes the coupled, quadratic BSDE system \eqref{def:bsde} and states the main results:  Theorem~\ref{thm:char} and Theorem~\ref{thm:existence}.  Given an $\sS^\infty\times\bmo$ solution to the BSDE system, Theorem~\ref{thm:char} constructs an equilibrium with limited participation and provides the verification arguments.  Theorem~\ref{thm:existence} proves the existence and uniqueness of an $\sS^\infty\times\bmo$ solution to the key BSDE system, and therefore proves the existence of an equilibrium.  Section~\ref{section:implications} studies the implications of our model in connection with the broader literature on equilibrium with limited participation.  In particular, we provide an example of a constant coefficient dividend stream and a Pareto efficient equilibrium comparison.  The proofs are contained in Section~\ref{section:proofs}.

\section{Model set-up}\label{section:set-up}

We study a continuous-time Radner equilibrium in a pure exchange economy with a fixed finite time horizon $T<\infty$.  We let $B=(B_t)_{t\in[0,T]}$ be a Brownian motion on the probability space $(\Omega,\sF,\bP)$ equipped with the augmented Brownian filtration $\bF=(\sF_t)_{t\in[0,T]}$. We assume that $\sF_T = \sF$.  Throughout this work, equality (and inequality) is assumed to hold $\bP$-a.s., and we suppress time from the notation when possible.  All stopping times are assumed to be $[0,T]$-valued.

The set of all adapted, continuous, and uniformly bounded processes is denoted by $\sS^\infty$.  For $q\in[1,\infty)$, the set of all progressively measurable processes $\zeta$ satisfying $\E[\int_0^T |\zeta_t|^qdt]<\infty$ is denoted by $\sL^q$.
  A martingale $M$ is said to be a {\it BMO-martingale} if there exists a constant $C>0$ such that for all stopping times $\tau$, we have
$$
  \E\left[\langle M\rangle_T-\langle M\rangle_\tau\,|\,\sF_\tau\right] \leq C. 
$$
In this case, we write $M\in\text{BMO}$.  If we need to stress that $M$ is a $\text{BMO}$ martingale with respect to a particular probability measure $\bQ$, we write $M\in\text{BMO}(\bQ)$.  Otherwise, we assume $\text{BMO}$ to mean the space $\text{BMO}(\bP)$.  The collection of progressively measurable processes $\sigma$ for which $\int \sigma dB$ is a $\text{BMO}$ martingale is denoted by $\bmo$, where we may also amend $\bmo$ to include a probability measure when extra emphasis is required.   
Background and details on $\text{BMO}$ martingales can be found in Kazamaki~\cite{Kaz94}.

The dividend stream process $D$ is an It\^o process with dynamics given by
\begin{equation}\label{def:state-process}
  dD_t = \mu_{D,t}dt + \sigma_{D,t}dB_t, \quad D_0\in\R.
\end{equation}
Throughout, we make the following assumption:
\begin{assumption}\label{asm:state-process}
  The processes $\mu_D$ and $\sigma_D$ are progressively measurable and there exists a constant $M>0$ such that for all $t\in[0,T]$ and $\omega\in\Omega$,
  $$
    \left|\mu_{D,t}(\omega)\right|\leq M \quad \text{and} \quad
    \frac{1}{M}\leq \sigma_{D,t}(\omega)\leq M.
  $$
\end{assumption}

\noindent{\bf The financial market.} 
The financial market consists of two securities:  a stock and an annuity.  Both securities are in one-net supply, and their prices are denominated in units of a single consumption good.  The annuity pays a constant dividend stream of $1$ per unit time and a lump sum of $1$ at time $T$.  The stock pays a dividend of $D_t$ per unit time and a lump sum of $D_T$ at time $T$.

The stock and annuity prices will be determined endogenously in equilibrium as continuous semimartingles $S=(S_t)_{t\in[0,T]}$ and $A=(A_t)_{t\in[0,T]}$, respectively.  Their terminal dividends are defined as
$$
  S_T = D_T \quad \text{and} \quad A_T = 1.
$$
We focus on models for which $A$ and $S$ have dynamics of the form
\begin{align}
  dA &= -dt + A\left(\mu_{A}\,dt+\sigma_{A}\,dB\right), \quad A_T=1,
  \label{def:A-dynamics} \\
  dS &= -D_t\,dt + \left(\mu_{S}+S\mu_{A}\right)dt + \left(\sigma_{S}+S\sigma_{A}\right)dB, \quad S_T = D_T.\label{def:S-dynamics}
\end{align}
The price processes and their dynamics are outputs of an equilibrium.  The coefficients $\mu_A$, $\sigma_A$, $\mu_S$, and $\sigma_S$ must be progressively measurable and will be proven to have sufficient regularity for \eqref{def:A-dynamics} and \eqref{def:S-dynamics} to be well-defined.

The annuity provides a constant dividend stream at all times $t\in[0,T]$.  It plays the role of a zero-coupon bond when a zero-coupon bond trades in an economy with only terminal consumption.  When the annuity's volatility is zero, the annuity can replicate a locally riskless security.

Positions held in the annuity over time are denoted $\theta = (\theta_t)_{t\in[0,T]}$, and positions held in the stock over time are denoted $\psi = (\psi_t)_{t\in[0,T]}$.  For given price processes $A$ and $S$ and a pair of positions $(\theta, \psi)$, the associated wealth process $X = (X_t)_{t\in[0,T]}$ is defined by
$$
  X := \theta A + \psi S.
$$
Consumption occurs at a rate $c = (c_t)_{t\in[0,T]}$ per unit time and in a lump sum at time $T$.

There are many possibilities when defining admissible trading strategies for exponential investors, such as described in Delbaen et.\,al.~\cite{6AP} and Biagini and Sirbu~\cite{BS12S}.  Here, we employ an approach analogous to Choi and Larsen~\cite{CL15FS} and Choi and Weston~\cite{CW22SIFIN}, who study equilibria with lump sum consumption, rather than the running and lump sum consumption considered here.  A progressively measurable process $\xi$ is called a \textit{state price deflator} if $\xi$ is strictly positive and $\left(\xi A+\int_0^\cdot \xi_u du\right)$ and $\left(\xi S+\int_0^\cdot \xi_u D_udu\right)$ are local martingales\footnote{We do not restrict the value of $\xi_0$ other than requiring $\xi_0>0$ to be $\sF_0$-measurable (a constant, $\bP$-a.s.).  We make this choice for ease of computation to allow the equilibrium state price deflator to be the unrestricted agent's marginal utility of optimal consumption without normalization.}.  
\begin{definition} Given a state price deflator $\xi$, we say that the strategies $(\theta,\psi,c)$ are \textit{admissible with respect to $\xi$} if $\theta$, $\psi$, and $c$ are progressively measurable and
\begin{enumerate}
  \item $\theta$ and $\psi$ are bounded, $\int_0^T |c_t|dt<\infty$, $\bP$-a.s., and $\E\left[\int_0^T |\xi_t c_t|dt \right]<\infty$;
  \item The associated wealth process $X := \theta A + \psi S$ is continuous and satisfies the self-financing condition:
    $$
      dX = \theta (dA + dt) + \psi (dS + D\,dt) - c\,dt;
    $$
  
\end{enumerate}
If $(\theta,\psi, c)$ is admissible with respect to $\xi$, then we write $(\theta,\psi, c)\in\sA(\xi, x)$, where $x\in\R$ is used to denote the initial wealth of the strategy:
$$
  x := \theta_{0-} A_0 + \psi_{0-} S_0.
$$
\end{definition}
Unlike in Choi and Larsen~\cite{CL15FS} but similar to Choi and Weston~\cite{CW22SIFIN}, the state price deflator $\xi$ is not agent-specific because here, the economy will only consist of two agents, one of which will be constrained to hold zero stock shares.  This constraint simplifies the constrained agent's admissible wealth processes but adds to her set of possible dual elements.  The local martingale constraints on $\xi$ and the dynamics \eqref{def:A-dynamics} and \eqref{def:S-dynamics} force $\xi$ to have dynamics of the form
\begin{equation}\label{eqn:xi-dyn}
  d\xi = -\xi(r\,dt+\kappa\,dB),
\end{equation}
where $r$ and $\kappa$ must satisfy
$$
  r = \mu_A -\sigma_A\kappa \quad \text{ and } \quad \mu_S = \kappa\sigma_S.
$$
If we were to allow for agent-specific state price deflators, then both agents' state price deflators would be of the form \eqref{eqn:xi-dyn}, but $\kappa$ would be a free parameter for the constrained agent.  We are able to get away with an economy-wide state price deflator because the $\bmo$ structure provides enough regularity for both agents' verification arguments to go through.
\ \\

\noindent{\bf Economic agents.} 
The economy consists of two (classes of) economic agents: an unconstrained agent and a constrained agent.  Agent $1$ is unconstrained and can trade in both the stock and the annuity.  Agent $1$ is endowed with $\theta_{1,0-}\in\R$ shares in the annuity and $\psi_{1,0-} :=1$ shares in the stock.  Agent $2$, however, is constrained by only having access to the annuity and not the stock.  Agent $2$ is endowed with $\theta_{2,0-}\in\R$ shares in the annuity and is constrained to always hold zero shares of stock.  Because the assets are in one-net supply, we assume that $\theta_{1,0-} + \theta_{2,0-} = 1$.

The agents are utility-maximizing and seek to maximize their expected utility from running consumption and terminal wealth.  Both agents have exponential utility functions with risk aversion coefficients $\alpha_1, \alpha_2>0$ and time preference parameters $\rho_1, \rho_2\geq 0$.  Given price processes $A$ and $S$, a state price deflator $\xi$, and initial wealths $x_1 := \theta_{1,0-}A_0 + \psi_{1,0-} S_0$ and $x_2:=\theta_{2,0-}A_0$, the agents seek to solve
\begin{align}
  &\sup_{(\theta, \psi, c)\in\sA(\xi, x_1)} \E\left[-\int_0^T\exp\left(-\rho_1 t-\alpha_1 c_t\right)dt - \exp\left(-\rho_1 T-\alpha_1 X_T\right)\right], \label{def:optimal-1}\\
  &\sup_{(\theta, 0, c)\in\sA(\xi, x_2)} \E\left[-\int_0^T\exp\left(-\rho_2 t-\alpha_2 c_t\right)dt - \exp\left(-\rho_2 T-\alpha_2 X_T\right)\right]. \label{def:optimal-2}
\end{align}
The agents' optimization problems differ in their risk aversions, time preferences for consumption, and agent $2$'s constraint of holding zero shares of stock.

In Basak and Cuoco~\cite{BC98RFS}'s limited participation equilibrium, the financial market consists of a bank account and a stock, and the second agent is constrained to only invest in the bank account or consume.  Here, rather than allow for a traded bank account, we opt for a traded annuity.  Trading an annuity is mathematically advantageous with exponential agents because the annuity allows for a dimension reduction when studying the value functions associated with \eqref{def:optimal-1} and \eqref{def:optimal-2}.  The value functions can be decomposed in the form
$$
  -\exp\left(-\alpha_i\left(\frac{X_{i,t}}{A_t} -  Y_{i,t}\right)\right),
$$
where $X_i$ is agent $i$'s optimal wealth process, $A$ is the annuity price process, and $Y_i$ is agent $i$'s certainty equivalent process.  The processes $Y_1$ and $Y_2$ will play an important role below in Theorems~\ref{thm:char} and \ref{thm:existence}, as they appear in the BSDE system used to construct an equilibrium.

\begin{remark}
  It would be possible to allow for more agents in our setting and also to allow for stochastic income streams, similar to Weston and \v{Z}itkovi\'c~\cite{WZ20FS}, as long as we have the same two classes of economic agents.  We choose to omit these additional features because the essential arguments and interesting model features are already present in this simpler setting, and we do not wish to introduce features that serve mainly to increase the notational burden.
\end{remark}

\noindent{\bf Equilibrium.} An equilibrium occurs when markets clear and the agents act optimally.  Definition~\ref{def:equilibrium} makes this notion precise.
\begin{definition}\label{def:equilibrium}  Strategies $(\theta_1, \psi_1, c_1)$ and $(\theta_2, 0, c_2)$ and continuous semimartingales $A$ and $S$ form a {\it Radner equilibrium with limited participation} if there exists a state price deflator $\xi$ such that 
\begin{enumerate}
  \item {\it Strategies are optimal:} $(\theta_1, \psi_1, c_1)$ and $(\theta_2, 0, c_2)$ are admissible with respect to $\xi$,  $(\theta_1, \psi_1, c_1)$ solves \eqref{def:optimal-1}, and $(\theta_2, 0, c_2)$ solves \eqref{def:optimal-2};
  \item {\it Markets clear:} For all $t\in[0,T]$, we have
  \begin{align*}
    \theta_{1,t} + \theta_{2,t} &= 1, \\
    \psi_{1,t} &= 1, \\
    c_{1,t}+c_{2,t} &= 1+ D_t.
  \end{align*}

\end{enumerate}
\end{definition}
The market clearing requirement plus agent $2$'s constraint requires that agent $1$ must optimally choose to hold one share of stock for all time.

\section{Main results}\label{section:main-results}
We construct an equilibrium with the help of an $\sS^\infty\times\bmo$ solution $\big((a,Z_a),(Y_1,Z_1),(Y_2,Z_2)\big)$ to the BSDE system:
\begin{align}
\begin{split}\label{def:bsde}
  da&= \sigma_DZ_a\, dB + \left(-\exp(-a)+\alpha_\Sigma\mu_D + \rho_\Sigma-\frac{\alpha_\Sigma\alpha_2}{2}\sigma_D^2Z_2^2-\frac{\alpha_\Sigma\alpha_1}{2}\sigma_D^2\left(1-Z_2-\frac{Z_a}{\alpha_\Sigma}\right)^2\right)dt,\\
  dY_1 &= \sigma_D Z_1\,dB + \left(-\frac{\rho_1}{\alpha_1} + \frac{1+a+\alpha_1Y_1}{\alpha_1 \exp(a)} - \frac{\alpha_1}{2}\sigma_D^2\left(1-Z_2-\frac{Z_a}{\alpha_\Sigma}\right)^2 + \alpha_1\sigma_D^2 Z_1\left(1-Z_2-\frac{Z_a}{\alpha_\Sigma}\right)\right)dt,\\
  dY_2 &= \sigma_D Z_2\,dB+\left(-\frac{\rho_2}{\alpha_2}+\frac{1 + a +\alpha_2Y_2}{\alpha_2 \exp(a)}+\frac{\alpha_2}{2}\sigma_D^2 Z_2^2\right)dt,\\
  a_T &= Y_{1,T} = Y_{2,T} = 0,
\end{split}
\end{align}
where $\alpha_\Sigma := \left(\frac{1}{\alpha_1}+\frac{1}{\alpha_2}\right)^{-1}$ and $\rho_\Sigma:= \alpha_\Sigma\left(\frac{\rho_1}{\alpha_1}+\frac{\rho_2}{\alpha_2}\right)$.

This three-dimensional BSDE system is quadratic and coupled.  Upon inspection, the two-dimensional subsystem with equations for $a$ and $Y_2$ form a fully coupled, quadratic autonomous system, while $Y_1$ decouples and can be solved for in terms of $\big((a,Z_a),(Y_2, Z_2)\big)$.  The coupling of $\big((a,Z_a),(Y_2, Z_2)\big)$ occurs even amongst the quadratic terms.  However, a linear transformation of the $a$ and $Y_2$ equations allows us to form an auxiliary two-dimensional, coupled BSDE system that is diagonally quadratic.  The diagonally quadratic form along with appropriate regularity and growth on the transformed system's driver will allow us to apply the results of Fan et.\,al.~\cite{FHT20wp} to obtain existence and uniqueness of a solution.  We establish existence and uniqueness of an $\sS^\infty\times\bmo$ solution in Theorem~\ref{thm:existence} below.

Given an $\sS^\infty\times\bmo$ solution $\big((a,Z_a),(Y_1,Z_1),(Y_2,Z_2)\big)$ to \eqref{def:bsde}, we conjecture that the equilibrium annuity and stock price processes are given by
\begin{equation}\label{def:A-S}
  A := \exp(a) \quad \text{ and } \quad
  S := A\left(D-\frac{a}{\alpha_\Sigma} - Y_1-Y_2\right).
\end{equation}
$A$ and $S$ will satisfy the dynamics conjectured in \eqref{def:A-dynamics} and \eqref{def:S-dynamics} for $\kappa$, $\sigma_A$, $\mu_A$, $\mu_S$, and $\sigma_S$ defined by
\begin{align}\label{eqn:coefficients}
\begin{split}
  \kappa&:= \alpha_1\sigma_D\left(1-\frac{Z_a}{\alpha_2}-Z_2\right),\\
  \sigma_A &:= \sigma_D Z_a,\\
  \mu_A &:= \alpha_\Sigma\mu_D + \rho_\Sigma+\frac{\sigma_A^2}{2}-\frac{\alpha_\Sigma}{2}\left(\alpha_2\sigma_D^2Z_2^2 + \frac{1}{\alpha_1}(\kappa-\sigma_A)^2\right), \\
  \sigma_S &:= A\sigma_D\left(1-\frac{Z_a}{\alpha_\Sigma}-Z_1-Z_2\right),\\
  \mu_S &:= \sigma_S\left(\sigma_A + \alpha_1\sigma_DZ_1 + \frac{\alpha_1\sigma_S}{A}\right).
\end{split}
\end{align}
While $\sigma_A$, $\mu_A$, $\sigma_S$, and $\mu_S$ appear in the conjectured dynamics in \eqref{def:A-dynamics} and \eqref{def:S-dynamics}, $\kappa$ does not.  $\kappa$ represents the market price of risk and will appear in the dynamics of the equilibrium state price deflator as
$$
  d\xi = -\xi\left(\left(\mu_A-\kappa\sigma_A\right)dt + \kappa dB\right).
$$

The following result proves the existence of an equilibrium in terms of a solution to the BSDE system~\eqref{def:bsde}.  All equilibrium parameters -- asset prices (and their dynamics), optimal strategies, and the state price deflator -- are described in terms of the input parameters and the BSDE system solution.
\begin{theorem}\label{thm:char} Suppose that Assumption \ref{asm:state-process} holds and there exists an $\sS^\infty\times\bmo$ solution $\big((a,Z_a),(Y_1,Z_1),(Y_2,Z_2)\big)$ to the BSDE system~\eqref{def:bsde}.  Then, there exists a Radner equilibrium with limited particpation with annuity and stock price processes given by \eqref{def:A-S}.  For $\kappa$, $\sigma_A$, $\mu_A$, $\sigma_S$, and $\mu_S$ defined in \eqref{eqn:coefficients}, $A$ and $S$ satisfy the dynamics conjectured in \eqref{def:A-dynamics} and \eqref{def:S-dynamics}.  Furthermore, the agents' optimal wealth processes, $X_1$ and $X_2$, are given for $t\in[0,T]$ by 
\begin{align}
  X_{1,t} &:= A_t\left(\theta_{1,0-}+\frac{S_0}{A_0} + \int_0^t \frac{1}{A_u}\left(-\frac{a_u}{\alpha_1}-Y_{1,u}+\mu_{S,u}-\sigma_{S,u}\sigma_{A,u}\right)du + \int_0^t \frac{\sigma_{S,u}}{A_u}dB_u\right),\label{def:X1}\\
  X_{2,t} &:= A_t\left(\theta_{2,0-} - \int_0^t \frac{1}{A_u}\left(\frac{a_u}{\alpha_2}+Y_{2,u}\right)du\right).\label{def:X2}
\end{align}
The agents' optimal consumption processes are given by $c_i := \frac{a}{\alpha_i}+\frac{X_i}{A}+Y_i$ for $i=1,2$.  Agent $1$'s optimal stock holdings are $\psi_1 = 1$, and the agents' optimal annuity holdings are $\theta_1:=\frac{X_1-S}{A}$ and $\theta_2:=\frac{X_2}{A}$.  The process $\xi$ given by
\begin{equation}\label{def:xi}
  \xi_t := \alpha_1 \exp\left(-\rho_1 t - \alpha_1 c_{1,t}\right), \quad t\in[0,T].
\end{equation}
is a state price deflator.  The strategies $(\theta_1, \psi_1, c_1)$ and $(\theta_2, 0, c_2)$ are admissible with respect to $\xi$, and $X_1$ and $X_2$ are their associated wealth processes, respectively. $(\theta_1, \psi_1, c_1)$ solves \eqref{def:optimal-1}, and $(\theta_2, 0, c_2)$ solves \eqref{def:optimal-2}.
\end{theorem}

\begin{remark}[Endogenously nondegenerate volatility]
In equilibrium problems where an asset in the financial market pays a dividend with nonzero volatility, degeneracy of that asset's volatility can become an issue.  The endogenous completeness literature studies such questions in equilibrium and derivatives pricing; see, for example, Anderson and Raimondo~\cite{AR08E}, Hugonnier et.\,al.~\cite{HMT12E}, Herzberg and Riedel~\cite{HR13JME}, Kramkov and Predoiu~\cite{KP14SPA}, and Schwartz~\cite{S17FS}.  Recently, Escauriaza et.\,al.~\cite{ESX21AAP} studied an incomplete Radner equilibrium with a terminal dividend-paying stock and terminal consumption.  Escauriaza et.\,al.~\cite{ESX21AAP} overcame a difficulty in proving the existence of a solution to their equilibrium-characterizing BSDE system because of a discontinuity that arose in the BSDE driver when degeneracy occurs in the volatility.  Their volatility process, much like in the present work, was endogenously determined as an output of equilibrium.  
In our model, the stock pays a dividend with a nonzero volatility.  However, we use a duality approach to prove agent $1$'s optimality, which circumvents potential difficulties associated with degenerate endogenous volatility.


\end{remark} 

The following result, Theorem~\ref{thm:existence}, is one of the two main results of the paper and proves the existence and uniqueness of an $\sS^\infty\times\bmo$ solution to~\eqref{def:bsde}.  Together with the other main result, Theorem~\ref{thm:char}, an $\sS^\infty\times\bmo$ solution to \eqref{def:bsde} proves the existence of an equilibrium in the limited participation economy.
\begin{theorem}\label{thm:existence}
Under Assumption~\ref{asm:state-process}, there exists a unique $\sS^\infty\times\bmo$ solution to the BSDE system~\eqref{def:bsde}.
\end{theorem}

\section{Equilibrium implications}\label{section:implications}
In this section, we discuss how our model fits into the broader literature on equilibrium with limited participation and financial equilibrium in general.
\subsection{No bubbles in equilibrium}
Suppose that we have a Radner equilibrium with limited participation with prices $A$ and $S$, state price density $\xi$, and optimal strategies $(\theta_1, \psi_1, c_1)$, $(\theta_2, 0, c_2)$ that are admissible with respect to $\xi$.

Hugonnier~\cite{H12JET} observed that bubbles exist in the stock market of the original model of Basak and Cuoco~\cite{BC98RFS}.  A stock market bubble means that the equilibrium stock price is not the same as (and is strictly greater than) the cost of replicating the stock's dividends.  Because Hugonnier~\cite{H12JET} and Basak and Cuoco~\cite{BC98RFS} focus on agents with utility that is finite only on the positive real line (and, in fact, only prove the existence of an equilibrium for an economy with two logarithmic agents), the state price density has a strict local martingale component.

In contrast, the present model considers exponential agents, whose utility is finite on the entire real line.  Utility maximization problems for utility functions that are finite on the entire real line are known to produce dual elements that are martingales, not strict local martingales; see, for example, Bellini and Frittelli~\cite{BF02MF}.  Here, the local martingale component of $\xi$ is a martingale, not a strict local martingale.  Lemma~\ref{lemma:mg} below proves that $\left(\xi S +\int_0^\cdot \xi_u D_udu\right)$ is a martingale, which gives us that
\begin{equation}\label{eqn:repl-S}
  S_t = \E\left[\int_t^T \frac{\xi_u}{\xi_t} D_u du + \frac{\xi_T}{\xi_t}D_T\right], \quad t\in[0,T].
\end{equation}
Therefore, the equilibrium stock price agrees with the stock's replication price, showing that a stock market bubble does not exist.  Similarly, we can apply Lemma~\ref{lemma:mg} to see that there is no annuity market bubble because
\begin{equation}\label{eqn:repl-A}
  A_t = \E\left[\int_t^T \frac{\xi_u}{\xi_t} du + \frac{\xi_T}{\xi_t}\right], \quad t\in[0,T].
\end{equation}

\subsection{Pareto efficient case}
Our model's incompleteness stems from the second agent's constraint.  The Pareto efficient case removes agent $2$'s constraint and studies when both agents have access to both assets.  For $i=1,2$, we denote $U_i(t,c) := -\exp(-\rho_i t - \alpha_i c)$ for $t\in[0,T]$, $c\in\R$.  We consider a representative agent with weight $\gamma>0$, whose utility function is 
$$
  U_\gamma(t,c) := \sup_{c_1+c_2\leq c} \left\{U_1(t,c_1)+\gamma U_2(t,c_2)\right\} = -\exp\left(-\rho_\Sigma t - \alpha_\Sigma c\right), \quad
  t\in[0,T],\ c\in\R,
$$
where $\alpha_\Sigma := \left(\frac{1}{\alpha_1}+\frac{1}{\alpha_2}\right)^{-1}$ and $\rho_\Sigma := \alpha_\Sigma\left(\frac{\rho_1}{\alpha_1}+\frac{\rho_2}{\alpha_2}\right)$.  In what follows, we use the superscript $PE$ to distinguish the Pareto efficient terms from the model input parameters, which are the same for the Pareto efficient and limited participation settings.

Under Assumption \ref{asm:state-process}, based on the aggregate supply $1+D$, the Pareto efficient equilibrium's state price density is given by
$$
  \xi^{PE}_t = \frac{\partial}{\partial c} U_\gamma(t, 1+D_t), \quad t\in[0,T].
$$
The Pareto efficient equilibrium asset prices are given by formulas that are analogous to \eqref{eqn:repl-S} and \eqref{eqn:repl-A}:
\begin{align*}
  S^{PE}_t &= \E\left[\int_t^T \frac{\xi^{PE}_u}{\xi^{PE}_t} D_u du + \frac{\xi^{PE}_T}{\xi^{PE}_t}D_T\right] \quad \text{ and } \quad
  A^{PE}_t = \E\left[\int_t^T \frac{\xi^{PE}_u}{\xi^{PE}_t} du + \frac{\xi^{PE}_T}{\xi^{PE}_t}\right].
\end{align*}
The dynamics for $\xi^{PE}$ are
\begin{align*}
  d\xi^{PE} &= -\xi^{PE}\left(\left(\rho_\Sigma+\alpha_\Sigma \mu_D - \frac12\alpha_\Sigma^2\sigma_D^2\right)dt + \alpha_\Sigma\sigma_D dB \right).
\end{align*}
Based on $\xi^{PE}$'s dynamics, we can identify the Pareto efficient equilibrium interest rate as $r^{PE} = \rho_\Sigma + \alpha_\Sigma\mu_D-\frac12\alpha_\Sigma^2\sigma_D^2$ and the Pareto efficient market price of risk as $\kappa^{PE} = \alpha_\Sigma\sigma_D$.

\subsection{Constant coefficient case}\label{section:constant}
As an example, we consider the constant coefficient case, where the dividend stream is given by
$$
  dD_t = \mu_D dt + \sigma_D dB_t, \quad D_0\in\R.
$$
Here, $\mu_D\in\R$ and $\sigma_D>0$ are constants.  This case can be solved explicitly by noticing that an $\sS^\infty\times\bmo$ solution to \eqref{def:bsde} exists for which $Z_a = Z_1 = Z_2=0$.  For $r:= \rho_\Sigma + \alpha_\Sigma\mu_D - \frac12\alpha_\Sigma\alpha_1 \sigma_D^2$, we have
$$
  da = (-\exp(-a)+r)dt, \quad a_T = 0.
$$
Thus, $\mu_A = r$, $\sigma_A = 0$, and
$$
  dA = (Ar -1)dt, \quad A_T = 1.
$$
Since $\sigma_A = 0$, the annuity does not have a Brownian ($dB$) term.  (In fact, the annuity is not stochastic at all.)  In this market, the annuity can replicate a bank account with the interest rate $r$.  Furthermore, we have $\mu_S = \alpha_1 A \sigma_D^2$,  $\sigma_S = A\sigma_D$, and so,
$$
  dS = -D dt + \left(\alpha_1 A\sigma_D^2+rS\right)dt + \sigma_D A\,dB, \quad S_T = D_T.
$$


Next, we compare the market parameters between the Pareto efficient and limited participation cases to see that
\begin{align}
\begin{split}\label{eqn:PE-comparisons}
  \kappa = \alpha_1\sigma_D &> \alpha_\Sigma\sigma_D = \kappa^{PE},\\
  r = \rho_\Sigma + \alpha_\Sigma\mu_D - \frac12\alpha_\Sigma\alpha_1\sigma_D^2 &< \rho_\Sigma + \alpha_\Sigma\mu_D-\frac12\alpha_\Sigma^2\sigma_D^2 = r^{PE}.
\end{split}
\end{align}
The superscript $PE$ denotes the Pareto efficient equilibrium outputs, as opposed to the limited participation equilibrium outputs. The comparisons in \eqref{eqn:PE-comparisons} are as expected, qualitatively, and agree with those in Basak and Cuoco~\cite{BC98RFS}.  In the limited participation economy, as compared to the Pareto efficient setting, the constrained agent must be enticed to hold a share of stock.  A premium, via the market price of risk, is needed to ensure that the constrained agent will demand a sufficient number of shares to clear the stock market.  The annuity market faces an increased demand in the limited participation economy compared to the Pareto efficient economy.  Thus, the limited participation interest rate is lower than the Pareto efficient interest rate, due to increased demand for the annuity.

\section{Proofs}\label{section:proofs}
Before beginning the proof of Theorem \ref{thm:char}, we define and establish regularity for several quantities.  We will take these definitions and properties as given in the proofs below.

Suppose that Assumption \ref{asm:state-process} holds and that $\big((a,Z_a),(Y_1,Z_1),(Y_2,Z_2)\big)$ is an $\sS^\infty\times\bmo$ solution to \eqref{def:bsde}.  Given this solution, we define $A$ and $S$ by \eqref{def:A-S}, and note that $A$ and $S$ are continuous semimartingales.  The terms $\kappa$, $\sigma_A$, $\mu_A$, $\sigma_S$, $\mu_S$ are defined by \eqref{eqn:coefficients}.  We let $X_1$ be defined by \eqref{def:X1}, and $X_2$ be defined by \eqref{def:X2}.  We define $\psi_{1}:= 1$, $\theta_{1}:=\frac{X_1-S}{A}$, $\theta_2:=\frac{X_2}{A}$, and $c_i := \frac{a}{\alpha_i}+\frac{X_i}{A}+Y_i$ for $i=1,2$.  Since $A$, $S$, $X_1$, and $X_2$ are continuous semimartingales, we have that
$$
  \theta_{1,0} = \frac{X_{1,0}-S_0}{A_0} = \theta_{1,0-} \quad \text{ and }
  \quad \theta_{2,0} = \frac{X_{2,0}}{A_0} = \theta_{2,0-},
$$
while $\psi_{1,0} = 1 = \psi_{1,0-}$.

Using that $a, Y_1, Y_2\in\sS^\infty$ and $Z_a, Z_1, Z_2\in\bmo$ with the definitions just above, we have that $A, X_2, \theta_2, c_2, \psi_1\in\sS^\infty$ and $\sigma_A, \kappa, \sigma_S\in\bmo$.  Moreover, Theorem 2.2 in Kazamaki~\cite{Kaz94} guarantees that for all $\text{BMO}$ martingales $M$, there exists a constant $C>0$ such that $\E[\exp(C\langle M\rangle_T)]<\infty$.  Therefore, $S, \mu_A, \mu_S, X_1, \theta_1, c_1 \in \sL^q$ for all $q\in[1,\infty)$.

\begin{lemma}\label{lemma:xi}
  Suppose that Assumption \ref{asm:state-process} holds and that there exists an $\sS^\infty\times\bmo$ solution $\big((a,Z_a),(Y_1,Z_1),(Y_2,Z_2)\big)$ to the BSDE system~\eqref{def:bsde}.  The process $\xi$ given by \eqref{def:xi} is a state price deflator, and $(\theta_1, \psi_1, c_1)$, $(\theta_2, 0, c_2)$ are admissible with respect to $\xi$.  $X_1$ defined in \eqref{def:X1} is the wealth process corresponding to $(\theta_1, \psi_1, c_1)$, and $X_2$ defined in \eqref{def:X2} is the wealth process corresponding to $(\theta_2, 0, c_2)$.
\end{lemma}
\begin{proof}
  Let $\big((a,Z_a),(Y_1,Z_1),(Y_2,Z_2)\big)$ be an $\sS^\infty\times\bmo$ solution to \eqref{def:bsde}, and define $r := \mu_A - \kappa \sigma_A$.  Then, $r\in\sL^q$ for all $q\in[1,\infty)$ and
$$
  d\xi = -\xi (r\,dt + \kappa\,dB), \quad \xi_0 = \alpha_1 \exp\left(-\alpha_1 c_{1,0}\right).
$$
Then,
\begin{align*}
  d\left(\xi A + \int_0^\cdot \xi_u du\right) &= -\xi A (\kappa-\sigma_A)dB,\\
  d\left(\xi S + \int_0^\cdot \xi_uD_u du\right) &= \xi \left(\sigma_S -S (\kappa-\sigma_A)\right)dB,
\end{align*}
which shows that $\left(\xi A + \int_0^\cdot \xi_u du\right)$ and $\left(\xi S + \int_0^\cdot \xi_uD_u du\right)$ are local martingales.  Since $\xi$ is progressively measurable and $\xi>0$, we conclude that $\xi$ is a state price deflator.

Next, we verify that $(\theta_1, \psi_1, c_1)$ and $(\theta_2, 0, c_2)$ are admissible with respect to $\xi$.  The processes $\theta_1, \psi_1, c_1, \theta_2, c_2$ are progressively measurable.  Using that $\theta_2, c_2, A\in\sS^\infty$, $\sigma_A, \sigma_S\in\bmo$, and $\mu_A, \mu_S, \theta_1, c_1, S\in\sL^q$ for all $q\in[1,\infty)$ gives us that $\theta_1$ and $\theta_2$ are $\left(A+\int_0^\cdot du\right)$-integrable, $\psi_1$ is $\left(S+\int_0^\cdot D_u du\right)$-integrable, and $c_1$ and $c_2$ are Lebesgue-integrable.

Since $\theta_1:= \frac{X_1-S}{A}$ and $\theta_2:=\frac{X_2}{A}$, we have
$$
  X_1 = \theta_1 A + S = \theta_1 A +\psi_1 S \quad \text{ and } \quad
  X_2 = \theta_2 A.
$$
Calculating the dynamics of $X_1$ using \eqref{def:X1} yields
\begin{align}\label{eqn:X1-dynamics}
\begin{split}
  d X_1 &= A\, d \left(\frac{X_1}{A}\right) + \frac{X_1}{A}dA + d\left<\frac{X_1}{A}, A\right> \\
  &= \left(\mu_S -\sigma_A\sigma_S - \frac{a}{\alpha_1}-Y_1\right)dt + \sigma_SdB + \frac{X_1}{A}\big(A\mu_Adt+A\sigma_AdB -dt\big) + \sigma_A\sigma_S dt \\
  &= \mu_S dt + \sigma_S dB - c_1 dt + X_1\left(\mu_Adt + \sigma_A dB\right)\\
  &= 1 (dS + D dt) + \theta_1 (dA + dt) - c_1 dt,
  \end{split}
\end{align}
while calculating the dynamics of $X_2$ using \eqref{def:X2} yields
\begin{align*}
  d X_2 &= A\, d \left(\frac{X_2}{A}\right) + \frac{X_2}{A}dA + d\left<\frac{X_2}{A}, A\right> \\
  &= -\left(\frac{a}{\alpha_2}+Y_2\right)dt + \frac{X_2}{A}\big(dA + dt - dt\big)\\
  &= \theta_2(dA + dt) - c_2 dt.
\end{align*}
Since $a, Y_2\in\sS^\infty$ and $d\theta_2 = -\frac{1}{A}\left(\frac{a}{\alpha_2}+Y_2\right)dt$, $\theta_2$ is bounded.  We see that $\theta_1$ is bounded by calculating its dynamics
\begin{align*}
  d\theta_1
  &= d\left(\frac{X_1}{A}\right) - d\left(\frac{S}{A}\right)\\
  &= \frac{1}{A}\left(-\frac{a}{\alpha_1} - Y_1-\frac{S}{A}+D\right)dt \\
  &= \frac{1}{A}\left(-\frac{a}{\alpha_1} - Y_1+\frac{a}{\alpha_\Sigma}+Y_1+Y_2\right)dt \\
  &= \frac{1}{A}\left(\frac{a}{\alpha_2}+Y_2\right)dt \\
  &= -d\theta_2.
\end{align*}
Since $\theta_1$, $\theta_2$, and $\psi_1 = 1$ are bounded, it remains to prove that
$$
  \E\left[\int_0^T |\xi_uc_{1,u}|du\right]<\infty \quad \text{ and } \quad
  \E\left[\int_0^T |\xi_uc_{2,u}|du\right]<\infty.
$$
To this end, we recall that $a\in\sS^\infty$, so $A=\exp(a)$ is continuous and uniformly bounded from above and below away from zero.  We define
$$
  M:= \exp\left(-\int_0^\cdot \frac{1}{A_u}du\right)\xi A.
$$ 
Then, the terms in $M$ are sufficiently regular to calculate its dynamics, and we have $dM = -M(\kappa-\sigma_A)dB$.  Since $\kappa-\sigma_A\in\bmo$, Theorem 3.1 in Kazamaki~\cite{Kaz94} guarantees that $M$ is a martingale and there exists $p>1$ for which $M$ satisfies the reverse H\"older inequality $(R_p)$.  That is, there exist constants $p>1$ and $C>0$ such that for all stopping times $\tau$,
$$
  \E\left[M^p_T\,|\,\sF_\tau\right] \leq C M^p_\tau.
$$
By Jensen's inequality, for all stopping times $\tau$, we have
\begin{align*}
  \E\left[M^p_\tau\right] &= \E\left[\left(\E\left[M_T\,|\,\sF_\tau\right]\right)^p\right] \leq \E\left[M^p_T\right]\leq C M^p_0.
\end{align*}
We define $q>1$ by $\frac{1}{p}+\frac{1}{q} = 1$ and notice that $\xi = \frac{M\exp\left(\int_0^\cdot \frac{1}{A_u}du\right)}{A}$ and that $\frac{\exp\left(\int_0^\cdot \frac{1}{A_u}du\right)}{A}\in\sS^\infty$.  By H\"older's inequality, there exists a constant $C>0$ such that for $i\in\{1,2\}$ and all stopping times $\tau$, 
\begin{align*}
  \E\left[\int_0^T \left|\xi_t c_{i,t}\right|dt\right] 
  &= \E\left[\int_0^T \frac{M_t\exp\left(\int_0^t \frac{1}{A_u}du\right)}{A_t} \left|c_{i,t}\right|dt\right] \\
  &\leq C \left(\E\left[\int_0^T M_t^p\right]\right)^{\frac{1}{p}} \left(\E\left[\int_0^T\left|c_{i,t}\right|^q\right]\right)^{\frac{1}{q}}\\
  &\leq C\left(\E\left[\int_0^T\left|c_{i,t}\right|^q\right]\right)^{\frac{1}{q}}.
\end{align*}
Since $c_1, c_2\in\sL^{q'}$ for all $q'\in[1,\infty)$, we have $\E\left[\int_0^T |\xi_uc_{1,u}|du\right]<\infty$ and $\E\left[\int_0^T |\xi_uc_{2,u}|du\right]<\infty$, as desired.

Therefore, $(\theta_1, \psi_1, c_1)$ and $(\theta_2, 0, c_2)$ are admissible with respect to $\xi$, and their associated wealth processes are $X_1$ and $X_2$, respectively.

\end{proof}

\begin{lemma}\label{lemma:mg}
  Suppose that Assumption \ref{asm:state-process} holds and that there exists an $\sS^\infty\times\bmo$ solution $\big((a,Z_a),(Y_1,Z_1),(Y_2,Z_2)\big)$ to the BSDE system~\eqref{def:bsde}.  For $\xi$ given by \eqref{def:xi},
  the following processes are martingales:
  $$
    \left(\xi A + \int_0^\cdot \xi_u du\right), \quad
    \left(\xi X_2 + \int_0^\cdot \xi_u c_{2,u}du\right), \quad \text{and } \quad 
    \left(\xi S + \int_0^\cdot \xi_u D_{u}du\right),
  $$
  and the collection of random variables $\left(\xi_\tau S_\tau\right)_\tau$, where $\tau$ ranges over all stopping times, is uniformly integrable.  
  
  Furthermore, if $c_1+c_2 = 1+D$ and $X_1 + X_2 = S + A$, then $\left(\xi X_1 + \int_0^\cdot \xi_u c_{1,u}du\right)$ is a martingale.
\end{lemma}

\begin{proof}
  Let $\big((a,Z_a),(Y_1,Z_1),(Y_2,Z_2)\big)$ be an $\sS^\infty\times\bmo$ solution to \eqref{def:bsde}.  For $r:=\mu_A - \kappa\sigma_A$, we see that $\xi$ has dynamics
$$
  d\xi = -\xi (r dt + \kappa dB).
$$
The processes $\left(\xi A + \int_0^\cdot \xi_u du\right)$, $\left(\xi X_2 + \int_0^\cdot \xi_u c_{2,u}du\right)$, and $\left(\xi S + \int_0^\cdot \xi_u D_{u}du\right)$ have local martingale dynamics:
  \begin{align*}
  d\left(\xi A + \int_0^\cdot \xi_u du\right) 
    &= -\xi A\left(\kappa-\sigma_A\right)dB, \\
  d\left(\xi X_2 + \int_0^\cdot \xi_u c_{2,u} du\right) 
    &= -\xi X_2\left(\kappa-\sigma_A\right)dB, \\
  d\left(\xi S + \int_0^\cdot \xi_u D_{u} du\right) 
    &= \xi \left(\sigma_S - S(\kappa-\sigma_A)\right)dB.
\end{align*}
It remains to show that $\left(\xi A + \int_0^\cdot \xi_u du\right)$, $\left(\xi X_2 + \int_0^\cdot \xi_u c_{2,u}du\right)$, and $\left(\xi S + \int_0^\cdot \xi_u D_{u}du\right)$ are martingales and not merely local martingales.  To this end, we recall that $a\in\sS^\infty$, so $A=\exp(a)$ is continuous and uniformly bounded from above and below away from zero.  As we saw in the proof of Lemma~\ref{lemma:xi}, we define
$$
  M:= \exp\left(-\int_0^\cdot \frac{1}{A_u}du\right)\xi A.
$$ 
Then, the terms in $M$ are sufficiently regular to calculate $M$'s dynamics, and we have $dM = -M(\kappa-\sigma_A)dB$.  Since $\kappa-\sigma_A\in\bmo$, Theorem 3.1 in Kazamaki~\cite{Kaz94} implies that $M$ is a martingale with $\E\left[\langle M\rangle_T^{\frac12}\right]<\infty$. We calculate that
\begin{align*}
d\left(\xi A + \int_0^\cdot \xi_u du\right) 
    &= \exp\left(\int_0^\cdot \frac{1}{A_u}du\right) dM, \\
  d\left(\xi X_2 + \int_0^\cdot \xi_u c_{2,u} du\right) 
    &= \frac{\exp\left(\int_0^\cdot \frac{1}{A_u}du\right)X_2}{A} dM. \\
\end{align*}
Since $\exp\left(\int_0^\cdot \frac{1}{A_u}du\right)$ and $\frac{\exp\left(\int_0^\cdot \frac{1}{A_u}du\right)X_2}{A}$ are in $\sS^\infty$, there exists a constant $C>0$ such that
\begin{align*}
  \E\left[\left\langle \xi A + \int_0^\cdot \xi_udu\right\rangle_T^{\frac12}\right]
    &\leq C\, \E\left[\langle M\rangle_T^{\frac12}\right]<\infty,\\
    \E\left[\left\langle \xi X_{2,u} + \int_0^\cdot \xi_uc_{2,u}du\right\rangle_T^{\frac12}\right]
    &\leq C\, \E\left[\langle M\rangle_T^{\frac12}\right]<\infty.
\end{align*}
The Burkholder-Davis-Gundy inequality implies that $\left(\xi A + \int_0^\cdot \xi_u du\right)$ and $\left(\xi X_2 + \int_0^\cdot \xi_u c_{2,u}du\right)$ are martingales.

Next, we show that the local martingale $\left(\xi S + \int_0^\cdot \xi_u D_u du\right)$ is a martingale.  Since $S$ and $D$ are unbounded, we cannot prove martingality of $\left(\xi S + \int_0^\cdot \xi_u D_u du\right)$ in the same manner as before.  Instead, we will show that the families of random variables
$$
\left(\xi_\tau S_\tau\right)_\tau \quad \text{and} \quad
\left(\int_0^\tau \xi_u D_u du\right)_\tau
$$
are uniformly integrable when indexed over all stopping times $\tau$.  This uniform integrability will imply the martingality we've claimed.

Theorem 3.4 in Kazamaki~\cite{Kaz94} guarantees that there exists $p>1$ for which $M$ satisfies the reverse H\"older inequality $(R_p)$.  That is, there exist constants $p>1$ and $C>0$ such that for all stopping times $\tau$,
$$
  \E\left[M^p_T\,|\,\sF_\tau\right] \leq C M^p_\tau.
$$
Then, by Jensen's inequality, for all stopping times $\tau$, we have
\begin{align}\label{eqn:M-Jensen}
  \E\left[M^p_\tau\right] &= \E\left[\left(\E\left[M_T\,|\,\sF_\tau\right]\right)^p\right] \leq \E\left[M^p_T\right]\leq C M^p_0.
\end{align}
We define $\eps:=\sqrt{p}-1$ so that $\eps>0$ and $(1+\eps)^2 = p$.  We define $q>1$ by $\frac{1}{1+\eps}+\frac{1}{q} = 1$.  We note that $\xi S = M\exp\left(\int_0^\cdot \frac{1}{A_u}du\right) \left(D-\frac{a}{\alpha_\Sigma}-Y_1-Y_2\right)$.  Then for all stopping times $\tau$, we have
\begin{align*}
  \E\left[\left|\xi_\tau S_\tau\right|^{1+\eps}\right] 
  &= \E\left[M^{1+\eps}_\tau \left(\exp\left(\int_0^\tau \frac{1}{A_u}du\right)\left|D_\tau-\frac{a_\tau}{\alpha_\Sigma}-Y_{1,\tau}-Y_{2,\tau}\right|\right)^{1+\eps}\right] \\
  &\leq \E\left[M^p_\tau\right]^{\frac{1}{1+\eps}} \E\left[\left(\exp\left(\int_0^\tau \frac{1}{A_u}du\right)\left|D_\tau-\frac{a_\tau}{\alpha_\Sigma}-Y_{1,\tau}-Y_{2,\tau}\right|\right)^{(1+\eps)q}\right]^{\frac{1}{q}} \\
  &\leq \left(C M^p_0\right)^{\frac{1}{1+\eps}} \E\left[\left(\exp\left(\int_0^\tau \frac{1}{A_u}du\right)\left|D_\tau-\frac{a_\tau}{\alpha_\Sigma}-Y_{1,\tau}-Y_{2,\tau}\right|\right)^{(1+\eps)q}\right]^{\frac{1}{q}}.
\end{align*}
Using the $\sS^\infty$ bounds for $a$, $Y_1$, and $Y_2$, and the reverse H\"older constants $C$ and $p$, there exists a constant $C'>0$ such that for all stopping times $\tau$, we have
$$
  \E\left[\left|\xi_\tau S_\tau\right|^{1+\eps}\right] 
  \leq C'\left(1+\E\left[\left|D_\tau\right|^{(1+\eps)q}\right]^{\frac{1}{q}}\right).
$$
Under Assumption \ref{asm:state-process}, $\mu_D$ is bounded and so, there exists $C>0$ (which is possibly different from the constant $C$ before) such that for all stopping times $\tau$,
$$
  |D_\tau|^{(1+\eps)q} \leq C\left(1+\left|\int_0^\tau \sigma_{D,u}dB_u\right|^{(1+\eps)q}\right).
$$
The process $\sigma_D$ is bounded by $M$ by Assumption \ref{asm:state-process}.  By the Burkholder-Davis-Gundy inequality, there exists a constant $C>0$ (which is possibly different from the constant $C$ before) such that for all stopping times $\tau$,
\begin{align*}
  \E\left[|D_\tau|^{(1+\eps)q}\right]
     &\leq C\left(1+\E\left[\left(\int_0^\tau \sigma_{D,u}^2 du\right)^{\frac{(1+\eps)q}{2}}\right]\right) \\
     &\leq C\left(1+\E\left[\left(\int_0^T \sigma_{D,u}^2du\right)^{\frac{(1+\eps)q}{2}}\right]\right) \\
     &\leq C\left(1+\left(M^2T\right)^{\frac{(1+\eps)q}{2}}\right).
\end{align*}
Since this bound does not depend on the stopping time $\tau$, we have that $(\xi_\tau S_\tau)_\tau$ is a uniformly integrable family of random variables.

We may apply similar estimates to $\left(\int_0^\tau \xi_uD_udu\right)_\tau$ by noticing that
$$
  \xi D = \left(\xi A\exp\left(-\int_0^\cdot\frac{1}{A_u}du\right)\right) \frac{D\exp\left(\int_0^\cdot\frac{1}{A_u}du\right)}{A} = M \frac{D\exp\left(\int_0^\cdot\frac{1}{A_u}du\right)}{A}.
$$
Thus, $\left(\xi S + \int_0^\cdot \xi_u D_{u} du\right)$ is a local martingale and $\left(\xi_\tau S_\tau\right)_\tau$ and $\left(\int_0^\tau \xi_u D_u du\right)_\tau$ are uniformly integrable, which implies that $\left(\xi S + \int_0^\cdot \xi_u D_{u} du\right)$ is a martingale.

Lastly, we assume that the market clearing assumptions hold:  $c_1+c_2 = 1+D$ and $X_1+X_2 = S+A$. Then, we have 
$$
  \xi X_1 + \int_0^\cdot \xi_u c_{1,u}du
  = \xi (S+A-X_2) + \int_0^\cdot \xi_u(1+D_u-c_{2,u})du,
$$
which shows that $\left(\xi X_1 + \int_0^\cdot \xi_u c_{1,u}du\right)$ is the linear combination of three martingales, and so $\left(\xi X_1 + \int_0^\cdot \xi_u c_{1,u}du\right)$ is a martingale.

\end{proof}

Lemma~\ref{lemma:mg} proved martingale properties for $A$, $S$, $X_1$, and $X_2$ when applied to $\xi$.  Next, we extend these martingale properties to arbitrary admissible wealth processes.  We remind the reader that the state price deflator $\xi$ from \eqref{def:xi} is defined in terms of $c_1$, which is agent $1$'s candidate optimal consumption stream.  However, the $c$ that appears in $(\theta, \psi, c)\in\sA(\xi, x)$ is arbitrary, so long as it satisfies the admissibility criteria.
\begin{lemma}\label{lemma:mg-any-wealth}
  Suppose that Assumption \ref{asm:state-process} holds and that there exists an $\sS^\infty\times\bmo$ solution $\big((a,Z_a),(Y_1,Z_1),(Y_2,Z_2)\big)$ to the BSDE system~\eqref{def:bsde}.   Let $(\theta,\psi,c)$ be admissible with respect to the state price deflator $\xi$ given by \eqref{def:xi}, and denote the associated wealth process by $X$.  Then $\left(\xi X + \int_0^\cdot \xi_u c_{u}du\right)$ is a martingale.
\end{lemma}

\begin{proof}
  Let $\big((a,Z_a),(Y_1,Z_1),(Y_2,Z_2)\big)$ be an $\sS^\infty\times\bmo$ solution to \eqref{def:bsde}.  By Lemma~\ref{lemma:xi}, $\xi$ is a state price deflator.  Suppose that $(\theta, \psi, c)$ are admissible with respect to $\xi$, and let the associated wealth process be denoted by $X$. We have
\begin{align*}
  d&\left(\xi X +\int_0^\cdot \xi_u c_u du\right)\\
  &= \xi\Big(\theta(dA+dt) + \psi(dS+D\,dt)-X(r\,dt + \kappa\,dB) - \kappa\left(\theta A\sigma_A+\psi(\sigma_S + S\sigma_A)\right)dt\Big) \\
  &= \xi\Big(X(\sigma_A-\kappa)dB
    + \psi\left((\mu_S-\kappa\sigma_S)dt + \sigma_S dB\right)\Big)\\
  &=\xi\Big(X(\sigma_A-\kappa)
    + \psi\sigma_S\Big)dB, \quad \text{ since $\mu_S = \kappa\sigma_S$.}
\end{align*}
Thus, the process $\big(\xi X + \int_0^\cdot \xi_uc_udu\big)$ is a local martingale, which we seek to show is a martingale.  We proceed by proving that the families of random variables:
$$
\left(\xi_\tau X_\tau\right)_\tau \quad \text{and} \quad
\left(\int_0^\tau \xi_u c_u du\right)_\tau
$$
are uniformly integrable when indexed over all stopping times $\tau$.  This uniform integrability will imply the martingality we've claimed.

We recall that $\theta$, $\psi$, $A$ are bounded and $X=\theta A + \psi S$.  Since $\xi$ is a (uniformly integrable) martingale and Lemma~\ref{lemma:mg} gives us that $(\xi_\tau S_\tau)_\tau$ is uniformly integrable, then $(\xi_\tau X_\tau)_\tau$ is a uniformly integrable family of random variables.

To see that $\left(\int_0^\tau \xi_u c_u du\right)_\tau$ is uniformly integrable, we recall that the definition of admissibility with respect to $\xi$ ensures that $\E\left[\int_0^T|\xi_uc_u|du\right]<\infty$.  Then, for all stopping times $\tau$,
$$
  \left|\int_0^\tau \xi_u c_u du\right| \leq \int_0^T \left|\xi_u c_u\right| du
$$
implies that $\left(\int_0^\tau \xi_u c_u du\right)_\tau$ is uniformly integrable.

Thus, we have shown that $\left(\xi X + \int_0^\cdot \xi_u c_u du\right)$ is a local martingale and the collections of random variables $\left(\xi_\tau X_\tau\right)_\tau$ and $\left(\int_0^\tau \xi_u c_u du\right)_\tau$ are uniformly integrable, from which we conclude that $\left(\xi X + \int_0^\cdot \xi_u c_u du\right)$ is a martingale.

\end{proof}

\begin{proof}[Proof of Theorem \ref{thm:char}] 
Let $\big((a,Z_a),(Y_1,Z_1),(Y_2,Z_2)\big)$ be an $\sS^\infty\times\bmo$ solution to \eqref{def:bsde}.  By Lemma \ref{lemma:xi}, $\xi$ defined in \eqref{def:xi} is a state price deflator, and the strategies $(\theta_1, \psi_1, c_1)$ and $(\theta_2, 0, c_2)$ are admissible with respect to $\xi$.  We proceed to verify each condition of a Radner equilibrium with limited participation.

\ \\ 
\noindent{\it Market clearing holds.} One can verify by using the dynamics of $\frac{X_1}{A}$ and $\frac{X_2}{A}$ that
\begin{align*}
  d\left(\frac{X_1+X_2}{A}\right)
  &= \frac{1}{A}\left\{\mu_S -\sigma_A\sigma_S-\frac{a}{\alpha_\Sigma}-Y_1-Y_2\right\}dt + \frac{\sigma_S}{A}dB\\
  &= \frac{1}{A}\left\{\mu_S -\sigma_A\sigma_S+\frac{S}{A}-D\right\}dt + \frac{\sigma_S}{A}dB\\
  &= d\left(\frac{S}{A}\right) = d\left(\frac{S+A}{A}\right).
\end{align*}
Since $X_1$ and $X_2$ are continuous and $\theta_{1,0-}+\theta_{2,0-}=1$, we have $X_{1,0}+X_{2,0} = A_0(\theta_{1,0-}+\theta_{2,0-})+\psi_{1,0-}S_0 = A_0+S_0$.  Thus,
\begin{equation}\label{eqn:wealth-clearing}
  \frac{X_1+X_2}{A} = \frac{S+A}{A}.
\end{equation}
Clearing in the consumption market also holds:
$$
  c_1+c_2 = \frac{a}{\alpha_\Sigma}+Y_1+Y_2 +\frac{X_1+X_2}{A} 
  = D-\frac{S}{A}+\frac{S+A}{A} = D+1.
$$
Finally, $\theta_1+\theta_2 = \frac{X_1}{A}-\frac{S}{A}+\frac{X_2}{A} = 1$ and $\psi_1=1$ gives us market clearing in Definition~\ref{def:equilibrium}, as desired.

\ \\ 

\noindent{\it Optimality for agent $1$.}  
We prove agent $1$'s optimality by a duality argument.  By Lemma~\ref{lemma:xi}, the strategies $(\theta_1, \psi_1, c_1)$ are admissible with respect to $\xi$, while $X_1$ defined in \eqref{def:X1} is the corresponding wealth process.  We denote
$$
  U(t,c) := -\exp\left(-\rho_1 t -\alpha_1 c\right), \quad t\in[0,T],\ c\in\R,
$$
and define the Fenchel-Legendre transform of $-U(t, -c)$ by
$$
  \widetilde U(t,y) := \sup_{c\in\R} \left\{ U(t,c) - yc\right\}, \quad t\in[0,T],\ c\in\R.
$$
The functions $U$ and $\widetilde U$ are related by
$U(t, c) = \widetilde U\left(t, \frac{\partial}{\partial c}U(t,c)\right) + c \frac{\partial}{\partial c}U(t,c)$ for all  $t\in[0,T]$, $c\in\R$.  Moreover, by the definition of $\xi$, we have that
$$
  \xi_T = \frac{\partial}{\partial c}U(t,X_{1,T}) \quad \text{ and } \quad
  \xi_t = \frac{\partial}{\partial c}U(t,c_{1,t}), \quad 
  t\in[0,T].
$$
Putting these calculations together gives us
\begin{equation}\label{eqn:duality}
  \widetilde U(T, \xi_T) + \xi_T X_{1,T} = U(T, X_{1,T}) \quad \text{ and } \quad 
  \widetilde U(t, \xi_t) + \xi_t c_{1,t} = U(t, c_{1,t}), \quad t\in[0,T].
\end{equation}
Now, let $(\theta, \psi, c)$ be strategies that are admissible with respect to $\xi$, with $\theta_{0-} = \theta_{1,0-}$ and $\psi_{0-}=1$.  Since the corresponding wealth process $X$ is continuous, we have $X_0 = X_{0-} = \theta_{0-}A_0 + \psi_{0-}S_0 = X_{1,0}$.  Then,
\begin{align*}
  \E&\left[U(T, X_{T}) + \int_0^T U(T, c_{t})dt\right] \\
  &\leq \E\left[\widetilde U(T,\xi_T)+\xi_TX_T +\int_0^T\left(\widetilde U(t,\xi_t)+\xi_tc_t\right)dt\right] \quad \text{by definition of $\widetilde U$ and $\xi>0$}\\
  &= \E\left[\widetilde U(T,\xi_T) +\int_0^T\widetilde U(t,\xi_t)dt\right] + \xi_0X_{1,0} \\
  & \quad \quad \quad 
    \quad \text{since $\left(\xi X + \int_0^\cdot \xi_uc_udu\right)$ is a martingale by Lemma \ref{lemma:mg-any-wealth}}\\
  &= \E\left[\widetilde U(T,\xi_T)+\xi_TX_{1,T} +\int_0^T\left(\widetilde U(t,\xi_t)+\xi_t c_{1,t}\right)dt\right] \\
  & \quad \quad \quad
    \quad \text{since $\left(\xi X_1 + \int_0^\cdot \xi_uc_{1,u}du\right)$ is a martingale by market clearing and Lemma \ref{lemma:mg}}\\ 
  &= \E\left[U(T, X_{1,T}) + \int_0^T U(T, c_{1,t})dt\right] \quad \text{by \eqref{eqn:duality}}.
\end{align*}
Therefore, the strategies $(\theta_1, \psi_1, c_1)$ and wealth $X_1$ are optimal for agent $1$'s problem, \eqref{def:optimal-1}.

\ \\

\noindent{\it Optimality for agent $2$.}  Let strategies $(\theta, 0,c)$ be admissible with respect to $\xi$ with $\theta_{0,-}=\theta_{2,0-}$.  Let $X$ be the corresponding wealth process.  We define $E := \exp\left(-\alpha_2\left(\frac{X}{A}+Y_{2}\right)\right)$ and
\begin{align*}
  V_t := -E_t\exp(-\rho_2t) - \int_0^t \exp(-\alpha_2 c_u-\rho_2u)du, \quad t\in[0,T].
\end{align*}
Notice that since $X$ is continuous at zero, we have $V_0 = -E_0 = \exp\left(-\alpha_2\left(\theta_{2,0-}+Y_{2,0}\right)\right)$.  The process $V$ has dynamics $dV_t = e^{-\rho_2 t}\left(\mu_{V,t} dt + \sigma_{V,t} dB_t\right)$ where $\sigma_{V} = \alpha_2 E \sigma_D Z_{2}$ and
$$
  \mu_{V} = -e^{-\alpha_2 c} + E\left\{\frac{\alpha_2}{A}\left(\frac{X}{A}+Y_2\right) + \frac{1+a-\alpha_2 c}{A}\right\}.
$$
Next, we apply Fenchel's inequality to $\mu_V$ for the function $x\mapsto -e^{-\alpha_2 x}$ for $x\in\R$, which says that for all $x\in\R$,
$$
  -e^{-\alpha_2 x} \leq \frac{y}{\alpha_2}\left(\log\left(\frac{y}{\alpha_2}\right)-1\right) + xy, \quad \text{for all $y>0$}.
$$
Choosing $y = \frac{\alpha_2 E}{A}$ and using that $A= \exp(a)$ and $E=\exp\left(-\alpha_2\left(\frac{X}{A}+Y_2\right)\right)$ gives us 
\begin{align*}
  \mu_V &= -e^{-\alpha_2 c} +E\left\{\frac{\alpha_2}{A}\left(\frac{X}{A}+Y_2\right) + \frac{1+a-\alpha_2 c}{A}\right\} \\
  &\leq \frac{E}{A}\left(\log\left(\frac{E}{A}\right)-1\right)+ \frac{\alpha_2 c E}{A} +   E\left\{\frac{\alpha_2}{A}\left(\frac{X}{A}+Y_2\right) + \frac{1+a-\alpha_2 c}{A}\right\} \\
  &= \frac{E}{A}\left\{-\alpha_2\left(\frac{X}{A}+Y_2\right)-a-1 + \alpha_2 c + \alpha_2\left(\frac{X}{A}+Y_2\right) + 1+a-\alpha_2c\right\} \\
  &= 0.
\end{align*}
The regularity of the terms in $\mu_V$ and $\sigma_V$ and $\mu_V\leq 0$ imply that $V$ is a supermartingale.

Since $Y_{2,T}=0$ and $A_T = \exp(a_T) = 1$, we have
\begin{align}\label{calc:ag2-optimality}
\begin{split}
  \E\left[\int_0^T -e^{-\rho_2t -\alpha_2 c_t}dt - e^{-\rho_2T-\alpha_2X_T}\right]
  & = \E[V_T] \\
  &\leq V_0 = -E_0 \\
  &= -\exp\left(-\alpha_2\left(\frac{X_0}{A_0}+Y_{2,0}\right)\right) \\
  &= -\exp\left(-\alpha_2\left(\theta_{2,0-}+Y_{2,0}\right)\right).
\end{split}
\end{align}

Next, we turn to the candidate optimal strategies and the candidate optimal wealth process for agent $2$.  By Lemma~\ref{lemma:xi}, the strategies $(\theta_2, 0, c_2)$ and corresponding wealth process $X_2$ are admissible with respect to $\xi$. 
%
Similar to above, we define $E_2 := \exp\left(-\alpha_2\left(\frac{X_2}{A}+Y_{2}\right)\right)$ and
\begin{align*}
  V_{2,t} := -E_{2,t}\exp(-\rho_2t) - \int_0^t \exp(-\alpha_2 c_{2,u}-\rho_2u)du, \quad t\in[0,T].
\end{align*}
Using the dynamics of $Y_2$ in \eqref{def:bsde} and $\frac{X_2}{A}$ in~\eqref{def:X2}, we see that $V_2$ is a local martingale with dynamics and initial condition
$$
  dV_{2,t} = \alpha_2 E_{2,t}e^{-\rho_2t}\sigma_{D,t} Z_{2,t}dB_t, \quad V_{2,0} = E_{2,0} = - \exp\left(-\alpha_2\left(\theta_{2,0-}+Y_{2,0}\right)\right).
$$
Since $a, Y_2 \in\sS^\infty$, $Z_2\in\bmo$, and $\sigma_D$ is progressively measurable and bounded, we have that $V_2$ is a martingale.  Recalling the calculation~\eqref{calc:ag2-optimality} above, we have that for any strategies $(\theta, 0, c)$ that are admissible with respect to $\xi$ with corresponding wealth process $X$ and $\theta_{0-}=\theta_{2,0-}$, we have
\begin{align*}
  \E\left[\int_0^T -e^{-\rho_2t -\alpha_2 c_t}dt - e^{-\rho_2T-\alpha_2X_T}\right]
  & = \E[V_T] \\
  &\leq V_0 \\
  &= -\exp\left(-\alpha_2\left(\theta_{2,0-}+Y_{2,0}\right)\right)\\
  &= V_{2,0}\\
  &= \E[V_{2,T}] \\
  &= \E\left[\int_0^T -e^{-\rho_2t -\alpha_2 c_{2,t}}dt - e^{-\rho_2T-\alpha_2X_{2,T}}\right].
\end{align*}
Therefore, the strategies $(\theta_2,0,c_2)$ with wealth $X_2$ are optimal for agent $2$.
\end{proof}

Next, we turn to the proof of Theorem~\ref{thm:existence}, which proves the existence and uniqueness of an $\sS^\infty\times\bmo$ solution to \eqref{def:bsde}.
\begin{proof}[Proof of Theorem \ref{thm:existence}] 
In order to prove the existence and uniqueness of a solution to \eqref{def:bsde}, we break the proof into steps.  First, we notice that agent $1$'s equation, $Y_1$, in BSDE system \eqref{def:bsde} decouples from the three-dimensional system.  We will initially prove the existence and uniqueness of an $\sS^\infty\times\bmo$ solution $\big((a,Z_a), (Y_2,Z_2)\big)$ to the two-dimensional BSDE system \eqref{def:bsde} with $Y_1$ excluded.  The $Y_1$ equation's solution will be constructed at the end of the proof.  In order to prove existence and uniqueness of the two-dimensional subsystem, we perform a truncation and linear transformation.  The truncated, transformed problem will be diagonally quadratic, allowing us to apply Theorem 2.4 from Fan et.\,al.~\cite{FHT20wp}.  Then, we develop estimates on the truncated problem in order to lift the truncation and, thus, prove existence and uniqueness for the two-dimensional BSDE subsystem of~\eqref{def:bsde}.  Finally, constructing a solution to $Y_1$'s equation will prove the existence of an $\sS^\infty\times\bmo$ solution to the original BSDE system~\eqref{def:bsde}.  We conclude the proof by showing that the solution to \eqref{def:bsde} is unique.\\

\noindent{\it Defining auxiliary two-dimensional truncated BSDE systems.} 
  For $N\geq 1$, we let $m_N(x):=\min(N,\max(-N,x))$, $x\in\R$.  We consider two truncated BSDE systems.  First, consider
\begin{align}\label{def:bsde-N}\tag{$BSDE_N$}
\begin{split}
  da&= \sigma_DZ_a\, dB + \left(-e^{-\max(a,-N)}+\alpha_\Sigma\mu_D + \rho_\Sigma-\frac{\alpha_\Sigma\alpha_2}{2}\sigma_D^2Z_2^2-\frac{\alpha_\Sigma\alpha_1}{2}\sigma_D^2\left(1-Z_2-\frac{Z_a}{\alpha_\Sigma}\right)^2\right)dt,\\
  dY_2 &= \sigma_D Z_2\,dB+\left(-\frac{\rho_2}{\alpha_2}+\frac{1}{\alpha_2}e^{-\max(a,-N)}\left(1+\alpha_2 m_N(Y_2)+ \max(a,-N)\right)+\frac{\alpha_2}{2}\sigma_D^2 Z_2^2\right)dt,\\
  a_T &= Y_{2,T} = 0.
\end{split}
\end{align}
The system \eqref{def:bsde-N} is a truncated version of the original two-dimensional BSDE system~\eqref{def:bsde}.  For $N\geq 1$, we also consider the BSDE system
\begin{align}\label{def:bsde-transform}\tag{$BSDE'_N$}
\begin{split}
  dY_\Sigma&= \sigma_DZ_\Sigma\, dB + \alpha_\Sigma\left(\mu_D + \frac{\rho_1}{\alpha_1}-\frac{\alpha_1}{2}\sigma_D^2\left(1-\frac{Z_\Sigma}{\alpha_\Sigma}\right)^2\right.\\
  &\quad\quad \left.+e^{-\max(Y_\Sigma-\alpha_\Sigma Y_2,-N)}\left(-\frac{1}{\alpha_\Sigma}+\frac{1}{\alpha_2}\left(1+ \alpha_2 m_N(Y_2)+\max(Y_\Sigma-\alpha_\Sigma Y_2,-N)\right)\right)\right)dt,\\
  dY_2 &= \sigma_D Z_2\,dB+\left(-\frac{\rho_2}{\alpha_2}+\frac{\alpha_2}{2}\sigma_D^2 Z_2^2\right.\\
  &\quad\quad \left.+\frac{1}{\alpha_2}e^{-\max(Y_\Sigma-\alpha_\Sigma Y_2,-N)}\left(1+\alpha_2 m_N(Y_2)+ \max(Y_\Sigma-\alpha_\Sigma Y_2,-N)\right)\right)dt,\\
  Y_{\Sigma,T} &= Y_{2,T} = 0.
\end{split}
\end{align}
The systems \eqref{def:bsde-N} and \eqref{def:bsde-transform} are related through a linear transformation.  In particular, for $N\geq 1$, given the existence and uniqueness of an $\sS^\infty\times\bmo$ solution $\big((Y_\Sigma, Z_\Sigma), (Y_2, Z_2)\big)$ to \eqref{def:bsde-transform}, there exists a unique $\sS^\infty\times\bmo$ solution $\big((a,Z_a),(Y_2,Z_2)\big)$ to \eqref{def:bsde-N}.  The two solutions are related by defining $a:= Y_\Sigma-\alpha_\Sigma Y_2$ and $Z_a:=Z_\Sigma -\alpha_\Sigma Z_2$.  Therefore, in order to prove the existence and uniqueness of an $\sS^\infty\times\bmo$ solution to \eqref{def:bsde-N}, it suffices to show the existence and uniqueness of an $\sS^\infty\times\bmo$ solution to \eqref{def:bsde-transform}. \\

\noindent{\it Existence and uniqueness of \eqref{def:bsde-N} and \eqref{def:bsde-transform}.}
To establish the existence and uniqueness of an $\sS^\infty\times\bmo$ solution to \eqref{def:bsde-transform}, we apply Theorem 2.4 in Fan et.\,al.~\cite{FHT20wp} to \eqref{def:bsde-transform} after first checking that the conditions are met.  The drivers for $Y_\Sigma$ and $Y_2$ in~\eqref{def:bsde-transform} are given, respectively, for $t\in[0,T]$, $\omega\in\Omega$, $y_\Sigma,y_2\in\R$,  $z_\Sigma, z_2\in\R$ by
\begin{align*}
  g_\Sigma&(t,\omega,y_\Sigma,y_2,z_\Sigma,z_2) \\
    &= \alpha_\Sigma\left(\mu_{D,t}(\omega) + \frac{\rho_1}{\alpha_1}-\frac{\alpha_1}{2}\sigma_{D,t}(\omega)^2\left(1-\frac{z_\Sigma}{\alpha_\Sigma}\right)^2\right.\\
  &\quad\quad \left.+e^{-\max(y_\Sigma-\alpha_\Sigma y_2,-N)}\left(-\frac{1}{\alpha_\Sigma}+\frac{1}{\alpha_2}\left(1+ \alpha_2 m_N(y_2)+\max(y_\Sigma-\alpha_\Sigma y_2,-N)\right)\right)\right)\\
  g_2&(t,\omega,y_\Sigma,y_2,z_\Sigma,z_2) \\
    &= -\frac{\rho_2}{\alpha_2}+\frac{1}{\alpha_2}e^{-\max(y_\Sigma-\alpha_\Sigma y_2,-N)}\left(1+\alpha_2 m_N(y_2)+ \max(y_\Sigma-\alpha_\Sigma y_2,-N)\right)+\frac{\alpha_2}{2}\sigma_{D,t}^2(\omega) z_2^2.
\end{align*}
By Assumption~\ref{asm:state-process}, $\mu_D$ and $\sigma_D$ are progressively measurable and uniformly bounded by $M>0$.  The functions $m_N$, $\exp(-\max(\cdot, -N))$, and
\begin{equation}\label{eqn:bdd-fn}
  x\mapsto e^{-\max(x,-N)}(1+\max(x,-N)), \quad x\in\R
\end{equation}
are bounded and Lipschitz.  The bounds and Lipschitz constants depend on $N$, which is okay at this stage in the proof.  In the notation of Fan et.\,al.~\cite{FHT20wp}, we take
\begin{align}
  \alpha_t(\omega) &:= \max\left(\alpha_\Sigma\left(\frac{\rho_1}{\alpha_1} + |\mu_{D,t}(\omega)|+ \alpha_1 M^2\right), \frac{\rho_2}{\alpha_2}\right), \quad t\in[0,T],\ \omega\in\Omega,\label{def:alpha}\\
  \phi &:= \exp(N)\left(1+N\left(1+\frac{1}{\alpha_2}\right)\max(\alpha_\Sigma,1)\right), \notag\\
  \gamma &:= \max\left(\frac{2\alpha_1}{\alpha_\Sigma}M^2, \alpha_2 M^2\right).\notag
\end{align}
Since $\alpha_\Sigma<\alpha_2$, we have $\phi > e^N\left(1+\alpha_\Sigma N\left(1+\frac{1}{\alpha_2}\right)\right)$.  Thus, hypothesis $(H1)$ from Fan et.\,al.~\cite{FHT20wp} holds since for $t\in[0,T]$, $\omega\in\Omega$, $y_\Sigma, y_2\in\R$, and $z_\Sigma, z_2\in\R$, we have
\begin{align}\label{eqn:g-estimates}
\begin{split}
  |g_\Sigma(t,\omega,y_\Sigma,y_2,z_\Sigma,z_2)| &\leq \alpha_t(\omega) + e^N\left(1+\alpha_\Sigma N\left(1+\frac{1}{\alpha_2}\right)\right) + \frac{\gamma}{2}z_\Sigma^2
    \leq\alpha_t(\omega) + \phi + \frac{\gamma}{2}z_\Sigma^2,\\
  |g_2(t,\omega,y_\Sigma,y_2,z_\Sigma,z_2)| &\leq \alpha_t(\omega)+N e^N\left(1+\frac{1}{\alpha_2}\right)+\frac{\gamma}{2}z_z^2
    \leq \alpha_t(\omega) + \phi + \frac{\gamma}{2}z_2^2.
  \end{split}
\end{align}

Next, we verify hypothesis $(H2)$ from Fan et.\,al.~\cite{FHT20wp}.  We use the Lipschitz properties of $m_N$, $\exp(-\max(\cdot, -n))$, and \eqref{eqn:bdd-fn} to see that there exists a constant $C>0$ such that for $t\in[0,T]$, $\omega\in\Omega$, $y_\Sigma, y_2, \ybar_\Sigma, \ybar_2\in\R$, and $z_\Sigma, z_2, \zbar_\Sigma, \zbar_2\in\R$, we have
\begin{align*}
  |&g_\Sigma(t,\omega,y_\Sigma,y_2,z_\Sigma,z_2)-g_\Sigma(t,\omega,\ybar_\Sigma,\ybar_2,\zbar_\Sigma,\zbar_2)|\\
  &\leq e^N\left(2+\alpha_\Sigma N\left(1+\frac{1}{\alpha_2}\right)\right)\left(|y_\Sigma-\ybar_\Sigma|+\alpha_\Sigma |y_2-\ybar_2|\right) + \frac{\alpha_1 M^2}{2\alpha_\Sigma^2}\left(2\alpha_\Sigma + |z_\Sigma|+|\zbar_\Sigma|\right)|z_\Sigma-\zbar_\Sigma|\\
    &\leq C\left(1+|z_\Sigma|+|\zbar_\Sigma|\right)\left(|y_\Sigma-\ybar_\Sigma|+|y_2-\ybar_2|+|z_\Sigma-\zbar_\Sigma|\right),
\end{align*}
and
\begin{align*}
  |&g_2(t,\omega,y_\Sigma,y_2,z_\Sigma,z_2)-g_2(t,\omega,\ybar_\Sigma,\ybar_2,\zbar_\Sigma,\zbar_2)|\\
  &\leq e^N\left(\frac{1}{\alpha_\Sigma}+N\left(1+\frac{1}{\alpha_2}\right)\right)\left(|y_\Sigma-\ybar_\Sigma|+\alpha_\Sigma|y_2-\ybar_2|\right)
    + \frac{\alpha_2 M^2}{2}\left( |z_2|+|\zbar_2|\right)|z_2-\zbar_2|\\
    &\leq C\left(1+|z_2|+|\zbar_2|\right)\left(|y_\Sigma-\ybar_\Sigma|+|y_2-\ybar_2|+|z_2-\zbar_2|\right).
\end{align*}
Hence, hypothesis $(H2)$ from Fan et.\,al.~\cite{FHT20wp} holds.  Hypothesis $(H3)$ from Fan et.\,al.~\cite{FHT20wp} holds because $Y_{\Sigma,T} = Y_{2,T}=0$ and $\alpha$, defined in \eqref{def:alpha}, is uniformly bounded.

To see that hypothesis $(H4)$ from Fan et.\,al.~\cite{FHT20wp} holds for $\lambda=0$, we see from \eqref{eqn:g-estimates} above that for $t\in[0,T]$, $\omega\in\Omega$, $y_\Sigma, y_2\in\R$, and $z_\Sigma, z_2\in\R$,
\begin{align*}
  \text{sgn}(y_\Sigma)g_\Sigma(t,\omega,y_\Sigma,y_2,z_\Sigma,z_2)&\leq |g_\Sigma(t,\omega,y_\Sigma,y_2,z_\Sigma,z_2)| \leq \alpha_t(\omega) + \phi + \frac{\gamma}{2}z_\Sigma^2,\\
  \text{sgn}(y_2)g_2(t,\omega,y_\Sigma,y_2,z_\Sigma,z_2)&\leq |g_2(t,\omega,y_\Sigma,y_2,z_\Sigma,z_2)| \leq \alpha_t(\omega) + \phi + \frac{\gamma}{2}z_2^2.
\end{align*}
Since hypotheses $(H1)$--$(H3)$ and $(H4)$ with $\lambda=0$ hold (using the $\lambda$ notation from Fan et.\,al.~\cite{FHT20wp}), we apply Theorem 2.4 
 from Fan et.\,al.~\cite{FHT20wp} to obtain a unique $\sS^\infty\times\bmo$ solution $\left((Y_\Sigma^{(N)},Z_\Sigma^{(N)}),(Y_2^{(N)},Z_2^{(N)})\right)$ to \eqref{def:bsde-transform}.  Using the linear transformation $a^{(N)}:= Y^{(N)}_\Sigma -\alpha_\Sigma Y^{(N)}_2$ and $Z_a^{(N)}:= Z^{(N)}_\Sigma -\alpha_\Sigma Z^{(N)}_2$, we have that $\left((a^{(N)},Z_a^{(N)}),(Y_2^{(N)},Z_2^{(N)})\right)$ is the unique $\sS^\infty\times\bmo$ solution to \eqref{def:bsde-N}.\\

\noindent{\it Estimates uniform in $N$.} 
The bounds on $a^{(N)}$ and $Y_2^{(N)}$, and the $\bmo$ norms of $Z_a^{(N)}$ and $Z_2^{(N)}$ depend on $N$.  We seek to remove some of this dependence in order to prove the existence of a solution to the two-dimensional untruncated subsystem~\eqref{def:bsde}.  For the remainder of the proof, we denote a universal constant by $C$.  For us, a universal constant refers to a value $C>0$ that does not depend on $N$.  We allow for $C$ to depend on $\alpha_1$, $\alpha_2$, $\rho_1$, $\rho_2$, and bounds on $\mu_D$ and $\sigma_D$, and we also allow for it to change from line to line or paragraph to paragraph.

First, we perform an estimate on $a^{(N)}$, providing a lower bound. Since the driver of $a^{(N)}$ is bounded by
$$
  g_\Sigma(t,\omega,y_\Sigma, y_2, z_\Sigma, z_2) -\alpha_\Sigma g_2(t,\omega,y_\Sigma, y_2, z_\Sigma, z_2)
  \leq \alpha_\Sigma \mu_{D,t}(\omega) + \rho_\Sigma
   \leq C,
$$ 
we see that $\left(a^{(N)}_t - \int_0^t\left(\alpha_\Sigma \mu_{D,u} + \rho_\Sigma\right)du\right)_{t\in[0,T]}$ is a supermartingale, and therefore,
\begin{equation}\label{eqn:aN-bound}
  a^{(N)}_t \geq \E\left[a^{(N)}_T-C(T-t)\,|\,\sF_t\right] \geq -CT, \quad t\in[0,T].
\end{equation}

To produce a bound on $Y_2^{(N)}$, we use the fact that there is only one quadratic term in $Y_2^{(N)}$'s driver.  This, and the driver's linearity in $y$, allow us to relax the $N$-dependence in the upper and lower bounds using a Gronwall's Lemma approach.

Since $Z_2^{(N)}\in\bmo$ and $\sigma_D$ is bounded, Theorem 2.3 in Kazamaki~\cite{Kaz94} implies that there exists a probability measure $\widetilde{\bP}$ under which
$$
  d\widetilde{B}_t = dB_t +\frac12 \sigma_{D,t}\, Z^{(N)}_{2,t}dt, \quad t\in[0,T],
$$
is a Brownian motion.  
Theorem 3.3 in Kazamaki~\cite{Kaz94} guarantees that $Z_2^{(N)}$ remains in $\bmo$ under $\widetilde\bP$.  Thus, $\int_0^\cdot \sigma_D Z_2^{(N)}d\widetilde{B}$ is a $\widetilde\bP$-martingale.  We also note that $e^{-x}(1+x)$ is nonnegative and bounded from above by a universal constant when $x\geq -C$.  Taking $\widetilde\bP$-expectations of $Y^{(N)}_{2,t}$ yields
\begin{align}\label{eqn:Y2-calc}
\begin{split}
  &\left|Y^{(N)}_{2,t}\right|\\
  &= \left|\widetilde\E\left[Y^{(N)}_{2,T} - \int_t^T\left(-\frac{\rho_2}{\alpha_2}+\frac{1}{\alpha_2}e^{-\max(a^{(N)}_u,-N)}\left(1+\alpha_2 m_N\left(Y^{(N)}_{2,u}\right)+ \max\left(a^{(N)}_u,-N\right)\right)\right)du\,|\,\sF_t\right]\right|\\
  &\leq \frac{1}{\alpha_2}\int_t^T \widetilde\E\left[\rho_2+
    e^{-\max(a^{(N)}_u,-N)}\left(\left|1+\max(a^{(N)}_u,-N)\right|
    +\alpha_2 \left|m_N\left(Y^{(N)}_{2,u}\right)\right|\right)\,|\,\sF_t\right]du\\
  &\leq C\left(1+\int_t^T y^{(N)}(u)du\right),
\end{split}
\end{align}
where $y^{(N)}(t):= \left\|Y^{(N)}_{2,t}\right\|_{\bL^\infty(\widetilde\bP)}=\left\|Y^{(N)}_{2,t}\right\|_{\bL^\infty(\bP)}$, $t\in[0,T]$.  Since $y^{(N)}(T)=0$ and $\left\|Y^{(N)}_2\right\|_{\sS^\infty} = \sup_{t\in[0,T]} y^{(N)}(t)$, Gronwall's inequality implies that $y^{(N)}(t)\leq Ce^{C(T-t)}\leq Ce^{CT}$.  Thus, $\left\|Y^{(N)}_2\right\|_{\sS^\infty}$ is bounded by a universal constant.

\ \\
\noindent{\it Solving the original two-dimensional BSDE system.}  
Armed with the universal bounds derived above, we choose $N_0$ sufficiently large so that for $N\geq N_0$, the solution \\$\left((a^{(N)},Z_a^{(N)}),(Y_2^{(N)},Z_2^{(N)})\right)$ to \eqref{def:bsde-N} satisfies $a^{(N)}_t\geq - N_0$ for all $t\in[0,T]$ and $\left\|Y^{(N)}_2\right\|_{\sS^\infty}\leq N_0$.  Therefore, for all $N\geq N_0$, the truncated terms in \eqref{def:bsde-N} can be removed so that $\left((a^{(N)},Z_a^{(N)}),(Y_2^{(N)},Z_2^{(N)})\right)$ solves
\begin{align}\label{def:bsde-2d}\tag{$BSDE,\,2D$}
\begin{split}
  da&= \sigma_DZ_a\, dB + \left(-\exp(-a)+\alpha_\Sigma\mu_D + \rho_\Sigma-\frac{\alpha_\Sigma\alpha_2}{2}\sigma_D^2Z_2^2-\frac{\alpha_\Sigma\alpha_1}{2}\sigma_D^2\left(1-Z_2-\frac{Z_a}{\alpha_\Sigma}\right)^2\right)dt,\\
  dY_2 &= \sigma_D Z_2\,dB+\left(-\frac{\rho_2}{\alpha_2}+\frac{1 + a +\alpha_2Y_2}{\alpha_2 \exp(a)}+\frac{\alpha_2}{2}\sigma_D^2 Z_2^2\right)dt,\\
   a_T &= Y_{2,T} = 0,
\end{split}
\end{align}
which is precisely the coupled two-dimensional subsystem in \eqref{def:bsde}.

Uniqueness for \eqref{def:bsde-2d} within $\sS^\infty\times\bmo$ is not immediately obvious, even though solutions to the truncated \eqref{def:bsde-N} are unique within $\sS^\infty\times\bmo$.  To see that $\sS^\infty\times\bmo$ solutions to \eqref{def:bsde-2d} are unique, we need to consider analogous estimates as in the \eqref{def:bsde-N} case.  To this end, let $\big((a,Z_a),(Y_2,Z_2)\big)$ satisfy \eqref{def:bsde-2d}.  Then, as in the analysis of solutions to \eqref{def:bsde-N}, $\left(a_t - \int_0^t\left(\alpha_\Sigma\mu_{D,u}+\rho_\Sigma\right)du\right)_{t\in[0,T]}$ is a supermartingale.  As in \eqref{eqn:aN-bound}, we have $a\geq -CT$.  The analogous calculation to \eqref{eqn:Y2-calc} for $Y_2$ shows that $Y_2$ is bounded by the same universal constant as $Y^{(N)}_2$ for $N\geq N_0$.  Thus, $\big((a,Z_a),(Y_2,Z_2)\big)$ satisfies the truncated BSDE system~\eqref{def:bsde-N} for all $N\geq N_0$.  By $\sS^\infty\times\bmo$ uniqueness for \eqref{def:bsde-N}, solutions to \eqref{def:bsde-2d} are also unique within $\sS^\infty\times\bmo$.

\ \\
\noindent{\it Solving for $Y_1$.}  Let $\big((a,Y_2), (Z_a, Z_2)\big)$ be an $\sS^\infty\times\bmo$ solution to~\eqref{def:bsde-2d}, which has been shown to exist above.  Since $Z_a, Z_2\in\bmo$ and $\sigma_D$ is bounded, Theorem 2.3 in Kazamaki~\cite{Kaz94} implies that there exists a probability measure $\overline{\bP}$ under which
\begin{equation}\label{eqn:P-bar}
  d\overline{B}_t = dB_t + \alpha_1\sigma_{D,t} \left(1-Z_{2,t}-\frac{Z_{a,t}}{\alpha_\Sigma}\right)dt, \quad t\in[0,T],
\end{equation}
is a Brownian motion.  
Theorem 3.3 in Kazamaki~\cite{Kaz94} guarantees that $Z_a$ and $Z_2$ remain in $\bmo$ under $\overline\bP$.  We define $A := \exp(a)$.  Since $a\in\sS^\infty$, $A$ is bounded from above and below away from zero.  

We define $Y_{1,t}$ for $t\in[0,T]$ as the continuous version of
\begin{align*}
  Y_{1,t} := -\overline\E\left[\int_t^T e^{-\int_t^u \frac{1}{A_{u'}}du'}\left(-\frac{\rho_1}{\alpha_1}+\frac{1+a_u}{\alpha_1A_u}-\frac12\alpha_1\sigma_{D,u}^2\left(1-Z_{2,u}-\frac{Z_{a,u}}{\alpha_\Sigma}\right)^2\right)du\,|\,\sF_t\right].\end{align*}
To see that $Y_1$ is in $\sS^\infty$, we first note that $a\in\sS^\infty$ and $A$ is bounded from below away from zero.  Since $Z_a, Z_2\in\bmo(\overline\bP)$ and $\sigma_D$ is bounded,  we then have 
\begin{equation*}\label{eqn:bmo-bound-1}
  \sup_{\tau} \overline\E\left[\int_\tau^T \sigma_{D,u}^2\left(1-Z_{2,u}-\frac{Z_{a,u}}{\alpha_\Sigma}\right)^2du\,|\,\sF_\tau\right]\leq C < \infty,
\end{equation*}
where the supremum is taken over stopping times $\tau$.  Therefore, $Y_1\in\sS^\infty$.  

The process $M$ given for $t\in[0,T]$ by
$$
  M_t := e^{-\int_0^t \frac{1}{A_{u'}}du'}Y_{1,t}- \int_0^t e^{-\int_0^u \frac{1}{A_{u'}}du'}\left(-\frac{\rho_1}{\alpha_1}+\frac{1+a_u}{\alpha_1A_u}-\frac12\alpha_1\sigma_{D,u}^2\left(1-Z_{2,u}-\frac{Z_{a,u}}{\alpha_\Sigma}\right)^2\right)du
$$
is a square-integrable $\overline\bP$-martingale.  Since the filtration is the augmented Brownian filtration, the martingale representation theorem guarantees that there exists a progressively measurable integrand $\Delta$ such that $M_t = \int_0^t \Delta_u d\overline{B}_u$ for $t\in[0,T]$.  Since $\sigma_D$ is bounded from below away from zero, we set $Z_{1,t}:= \frac{\Delta_t\exp\left(\int_0^t \frac{1}{A_u}du\right)}{\sigma_{D,t}}$ for $t\in[0,T]$.  The $\overline\bP$-square integrability of $M$ gives us that $\overline\E\left[\int_0^T \exp\left(-2\int_0^t \frac{1}{A_u}du\right)\sigma_{D,u}^2Z_{1,u}^2du\right]<\infty$. Since $\overline\bP\sim\bP$, we have that $\int_0^T \exp\left(-2\int_0^t \frac{1}{A_u}du\right)\sigma_{D,u}^2Z_{1,u}^2du<\infty$, $\bP$-a.s., and likewise, $\int_0^T Z_{1,u}^2du<\infty$, $\bP$-a.s.

The dynamics of $M$ can be expressed two ways as $dM = \exp\left(-\int_0^\cdot \frac{1}{A_u}du\right)\sigma_D Z_1 d\overline B$ and
\begin{align*}
  dM 
  &= d\left(\exp\left(-\int_0^t \frac{1}{A_u}du\right) Y_1\right) \\
  &\quad\quad  - \exp\left(-\int_0^t \frac{1}{A_u}du\right)\left(-\frac{\rho_1}{\alpha_1}+\frac{1+a}{\alpha_1 A} - \frac12\alpha_1\sigma_D^2\left(1-Z_2-\frac{Z_a}{\alpha_\Sigma}\right)^2\right) dt\\
  &= \exp\left(-\int_0^t \frac{1}{A_u}du\right) dY_1 \\
  &\quad\quad  - \exp\left(-\int_0^t \frac{1}{A_u}du\right)\left(-\frac{\rho_1}{\alpha_1}+\frac{1+a+\alpha_1Y_1}{\alpha_1 A} - \frac12\alpha_1\sigma_D^2\left(1-Z_2-\frac{Z_a}{\alpha_\Sigma}\right)^2\right) dt.
\end{align*}
Solving for $dY_1$ gives us that
\begin{align*}
  dY_1
  &= \left(-\frac{\rho_1}{\alpha_1} + \frac{1+a+\alpha_1Y_1}{\alpha_1 A} - \frac12\alpha_1\sigma_D^2\left(1-Z_2-\frac{Z_a}{\alpha_\Sigma}\right)^2\right)dt
    + \sigma_D Z_1 d\overline B\\
  &= \left(-\frac{\rho_1}{\alpha_1} + \frac{1+a+\alpha_1Y_1}{\alpha_1 A} - \frac12\alpha_1\sigma_D^2\left(1-Z_2-\frac{Z_a}{\alpha_\Sigma}\right)^2
    + \alpha_1\sigma_D^2 Z_1 \left(1-Z_2-\frac{Z_a}{\alpha_\Sigma}\right)\right)dt \\
  &\quad\quad +\sigma_D Z_1 dB.
\end{align*}
Therefore, $\big((a,Z_a),(Y_1,Z_1),(Y_2,Z_2)\big)$ is a solution to \eqref{def:bsde}.  However, we claimed to have an $\sS^\infty\times\bmo$ solution, and it remains to show that $Z_1\in\bmo$.

\ \\
\noindent{\it Establishing $\bmo$ regularity of $Z_1$.}  We let $\phi(x):=-\exp(-x)$ for $x\in\R$, and let $\tau$ be a stopping time.  Applying It\^o's Lemma to $\phi(Y_1)$ under $\overline\bP$ and using that $\phi'$ is nonnegative while $a, Y_1\in\sS^\infty$ gives
\begin{align*}
  \phi&(Y_{1,T}) -\phi(Y_{1,\tau})\\
  &= \int_\tau^T \phi'(Y_{1,u})\sigma_{D,u}Z_{1,u}d\overline{B}_u
     + \frac12 \int_\tau^T \phi''(Y_{1,u})\sigma_{D,u}^2Z_{1,u}^2du \\
  &\quad\quad   + \int_\tau^T\left\{\phi'(Y_{1,u})\left(-\frac{\rho_1}{\alpha_1} + \frac{1+a_u+\alpha_1Y_{1,u}}{\alpha_1 A_u} - \frac12\alpha_1\sigma_{D,u}^2\left(1-Z_{2,u}-\frac{Z_{a,u}}{\alpha_\Sigma}\right)^2\right)\right\}du \\
  &\leq \int_\tau^T \phi'(Y_{1,u})\sigma_{D,u}Z_{1,u}d\overline{B}_u
     + \frac12 \int_\tau^T \phi''(Y_{1,u})\sigma_{D,u}^2Z_{1,u}^2du  + C.
\end{align*}
Since $Z_1$ is $\overline\bP$-square integrable, $\sigma_D$ is bounded, and $\phi'(Y_1)\in\sS^\infty$, we have that $\int_0^\cdot \phi'(Y_1)\sigma_D Z_1 d\overline{B}$ is a $\overline\bP$-martingale.  Moreover, since $Y_1\in\sS^\infty$, and $\phi''$ is negative and continuous, we may rearrange terms and compute $\overline\bP$-conditional expectations to obtain
\begin{align*}
  \overline\E\left[\int_\tau^T \sigma_{D,u}^2 Z_{1,u}^2du\,|\,\sF_\tau\right] \leq C.
\end{align*}
Recall that $C$ is a universal constant that does not depend on the choice of stopping time $\tau$, and $\sigma_D$ is bounded from below away from zero.  Thus, $Z_1\in\bmo(\overline\bP)$.  Theorem 3.3 in Kazamaki~\cite{Kaz94} guarantees that $Z_1\in\bmo(\bP)$, which completes the proof that $\big((a,Z_a),(Y_1,Z_1),(Y_2,Z_2)\big)$ is an $\sS^\infty\times\bmo$ solution to \eqref{def:bsde}.

\ \\
\noindent{\it Uniquenss for $(Y_1, Z_1)$.} Let $\big((a,Z_a,),(Y_2,Z_2)\big)$ be the unique $\sS^\infty\times\bmo$ solution to \eqref{def:bsde-2d}, and suppose that $\big((a,Z_a,),(Y_1,Z_1),(Y_2,Z_2)\big)$ and $\big((a,Z_a,),(\widetilde Y_1,\widetilde Z_1),(Y_2,Z_2)\big)$ are $\sS^\infty\times\bmo$ solutions to \eqref{def:bsde}.  We calculate the dynamics of under the probability measure $\overline\bP$ with Brownian motion $\overline B$ defined in \eqref{eqn:P-bar}, which gives us
$$
  d\left(Y_1-\widetilde Y_1\right)
  = \sigma_D \left(Z_1-\widetilde Z_1\right)d\overline B + \left(\frac{Y_1-\widetilde Y_1}{A}\right)dt.
$$
Using that $Y_{1,T}=\widetilde Y_{1,T}$ and taking $\overline\bP$-conditional expectations for $t\in [0,T]$ yields $Y_{1,t}-\widetilde Y_{1,t} = -\overline\E\left[\int_t^T \frac{Y_{1,u}-\widetilde Y_{1,u}}{A_u} du\,|\,\sF_t\right]$.  Since $A$ is bounded from below away from zero by a constant $1/C$, we have
$$
  \overline\E\left[\left|Y_{1,t}-\widetilde Y_{1,t}\right|\right]
  \leq C \int_t^T \overline\E\left[\left|Y_{1,u}-\widetilde Y_{1,u}\right| \right] du.
$$ 
Gronwall's Lemma gives us that $\overline\E\left[\left|Y_{1,t}-\widetilde Y_{1,t}\right|\right] = 0$.  Since $Y_1= \widetilde Y_1$, we have $\int_0^\cdot \sigma_D Z_1 dB = \int_0^\cdot \sigma_D \widetilde Z_1 dB$, which implies that
$$
  \int_0^T \sigma_{D,u}^2\left(Z_{1,u}-\widetilde Z_{1,u}\right)^2 du 
  = \langle \int_0^\cdot \sigma_D\left(Z_1-\widetilde Z_1\right)dB\rangle_T
  = 0.
$$
Since $\sigma_D$ is bounded from below away from zero, we have $Z_1=\widetilde Z_1$. 
Thus, $\big((a,Z_a),(Y_1,Z_1),(Y_2,Z_2)\big)$ is the unique $\sS^\infty\times\bmo$ solution to the BSDE system \eqref{def:bsde}.

\end{proof}

\bibliographystyle{plain}
\bibliography{finance_bib}

\end{document}